\documentclass{amsart}
\usepackage[all]{xy}                    

\CompileMatrices                            

\UseTips                                    

\input xypic
\usepackage[bookmarks=true]{hyperref}       
\usepackage[margin=1in]{geometry}
\usepackage[utf8]{inputenc}

\usepackage{amsmath, amsthm, amssymb, amsfonts, mathtools, bbold, stmaryrd} 
\usepackage{setspace}
\usepackage{graphicx}\graphicspath{{gfx/}} 
\usepackage{subcaption}

\usepackage[dvipsnames]{xcolor}
\usepackage{stackengine}
\usepackage{float}
\usepackage{enumerate}
\usepackage{enumitem}
\usepackage{tikz}
\usepackage{multirow,bigdelim}
\usepackage{algorithm}
  \usepackage{braket}
\usepackage{comment}
\usepackage{float} 
\usepackage{tikz-cd}
\usepackage[utf8]{inputenc}
\usepackage[T1]{fontenc}
\usepackage{mdframed}
\usepackage[english]{babel}
\usepackage{newlfont}
\usepackage{fancyhdr}
\usepackage{graphicx}
\usepackage{latexsym}

\usepackage{nicefrac}
\usepackage{pgf,tikz}
\usepackage{mathrsfs}
\usetikzlibrary{arrows}
\usepackage{listings}
\usepackage{spverbatim}

\definecolor{dkgreen}{rgb}{0,0.6,0}
\definecolor{gray}{rgb}{0.5,0.5,0.5}
\definecolor{mauve}{rgb}{0.58,0,0.82}

\reversemarginpar

\makeatletter
\newcommand{\skipitems}[1]{%
	\addtocounter{\@enumctr}{#1}%
}

\makeatother

\DeclareMathOperator{\Res}{Res}

\theoremstyle{plain}
\newtheorem{theorem}{Theorem}[section]
\newtheorem*{theorem*}{Theorem}

\newtheorem{corollary}[theorem]{Corollary}
\newtheorem{lemma}[theorem]{Lemma}

\theoremstyle{definition}
\newtheorem{definition}[theorem]{Definition}

\newtheorem{notation}[theorem]{Notation}

\newtheorem{proposition}[theorem]{Proposition}

\theoremstyle{plain}
\newtheorem{example}{Example}[section]
\theoremstyle{definition}
\newtheorem{remark}{Remark}[section]

\renewcommand{\phi}{\varphi}
\renewcommand{\theta}{\vartheta}
\renewcommand{\epsilon}{\varepsilon}

\renewcommand{\to}[1][]{\xrightarrow{\ #1\ }}

\newcommand{\old}[1]{}

\newtheorem{claim}{Claim}[theorem]

\title{An algorithm for the non-identifiability of rank-3 tensors}

\date{}

\author[]{Pierpaola Santarsiero}
\address{
	Pierpaola Santarsiero\\
	Universität Leipzig-MPI-MiS Leipzig \\ Germany
}
\email{pierpaola.santarsiero@mis.mpg.de \vspace{.3cm} }

\address{Edoardo Ballico\\ University of Trento \\ Italy}
\email{edoardo.ballico@unitn.it}

\address{Alessandra Bernardi\\ University of Trento \\ Italy}
\email{alessandra.bernardi@unitn.it}

\makeatletter
\@namedef{subjclassname@2020}{%
  \textup{2020} Mathematics Subject Classification}
\makeatother

\begin{document}
	
\begin{abstract}
We present an algorithm aimed to recognize if a given tensor is a  non-identifiable rank-3 tensor. 
\end{abstract}

	\maketitle
\begin{center}
	(With an appendix together with E. Ballico and A. Bernardi)
\end{center}

\section{Introduction}
	Over the last 60 years multilinear algebra made its way in the applied sciences. As a consequence, tensors acquired an increasingly central role in the applications and the problem of tensor rank decomposition has started to be studied by several non-mathematical communities (cf. e.g. \cite{AMR}, \cite{Kolda}, \cite{APRS}, \cite{BCquantum}, \cite{RS}, \cite{Rcit}, \cite{AetAl},  \cite{BBCMM}, \cite{BClibro}). 
	
	Fix $\mathbb{C}$-vector spaces $V_1,\dots,V_k$ of dimensions $n_1,\dots,n_k$ respectively. A tensor $T\in V_1\otimes \cdots \otimes V_k$ is called \emph{elementary} if $T=v_1\otimes \cdots \otimes v_k$ for some $v_i\in V_i$ with $i=1,\dots,k$. Elementary tensors are the building blocks of the tensor rank decomposition and the \emph{rank} $r(T)$ of a tensor $T\in V_1\otimes \cdots \otimes V_k$ is the minimum integer $r$ such that we can write $T$ as a combination of $r$ elementary tensors:
	$$
	T=\sum_{i=1}^{r}  v_{1,i}\otimes \cdots \otimes v_{k,i}, \hbox{ where all }v_{j,i}\in V_j \hbox{ for }j=1,\dots,k. 
	$$
	A rank-$r$ tensor $T $ is \emph{identifiable} if admits a unique rank decomposition up to reordering the elementary tensors and up to scalar multiplication. Remark that since the notion of rank does not depend on scalar multiplication, it is well defined for projective classes of tensors too.
	
 The first modern contribution on identifiability of tensors has been given by J. B. Kruskal \cite{K77} and, starting from Kruskal's result, over the years there have been many contributions on the identifiability problem (cf. e.g. \cite{CC2}, \cite{CO}, \cite{Bocci2011}, \cite{BCO}, \cite{COV14}, \cite{COV17}, \cite{BBC}, \cite{MMS}, \cite{GM19}, \cite{HOOShomotopy}, \cite{CM21}, \cite{CM22}, \cite{LMR22}).
	In particular, working in the applied fields, one may also be interested in the identifiability of specific tensors. Indeed, when translating an applied problem into the language of tensors one may be forced to deal with a very specific tensor that has a precise structure by reasons related to the nature of the applied problem itself. Working with specific tensors, the literature review becomes more scattered and most of the results can be considered extensions or generalizations of Kruskal's result (cf. \cite{Boralevi}, \cite{DD1}, \cite{DL14}, \cite{SDL15}, \cite{BV} and \cite{Lovitz}).
	
	The first complete classification on the identifiability problem appeared in \cite{BBS} where, together with E. Ballico and A. Bernardi, we completely characterize all identifiable tensors of rank either 2 or 3. The classification is based on the classical Concision Lemma (cf. \cite[Prop. 3.1.3.1]{Lands} and also Subsection \ref{Subsec:Concision} below) and, in particular, for  $r=2$ it has been proved that the only non-identifiable rank-$2$ tensors are $2\times 2$ matrices (cf. \cite[Proposition 2.3]{BBS}). A more interesting situation occurs for the rank-3 case, where there have been found $6$ different families of non-identifiable concise rank-3 tensors (cf. \cite[Theorem 7.1]{BBS}). 
	
	In this manuscript we present an algorithm aimed to recognize if a given tensor falls into one of the $6$ families above mentioned or not.
	
	The paper is organized as follows. Section \ref{Sec:Prel} is devoted to recollect basic notions needed to develop the algorithm. We start by recalling \cite[Theorem 7.1]{BBS} and explaining each case of the classification working in coordinates. In Subsection \ref{Subsec:Concision} we recall the coordinate description of the concision process for a tensor while Subsection \ref{Subsec:Pencil} is devoted to review basic facts on matrix pencils.
	In Section \ref{Sec:Algoritmo} is presented the algorithm itself. In particular, Subsection \ref{Subsec:3fattori} focuses on the 3-factors case, while Subsection \ref{Subsec:+3fattori} considers the general case of $k\geq 4$ factors.
	
	We end the manuscript with an appendix written together with E. Ballico and A. Bernardi in which we fix an imprecision in the statement of \cite[Proposition 3.10]{BBS} and consequently in an item in \cite[Theorem 7.1]{BBS}. In the following, if needed, we will refer to the correct statement of \cite[Proposition 3.10 and Theorem 7.1]{BBS} given in the forthcoming Proposition \ref{prop:new statement} and Theorem \ref{newMAINTHM} respectively.
	

	\section{Preliminary notions}\label{Sec:Prel}
		In the following we will work with tensors over $\mathbb{C}$.

		\begin{notation}
			Fix $k$ vector spaces $V_1,\dots,V_k$ of dimension $n_1+1,\dots,n_k+1$ respectively and let $N=\prod_{i=1}^k (n_i+1)-1$. By $\nu$ we denote the Segre embedding 
				\begin{align*}
				\nu: \mathbb{P}V_1\times \cdots \times \mathbb{P}V_k &\longrightarrow  \mathbb{P}(V_1 \otimes \cdots \otimes V_k)=\mathbb{P}^N\\
				([v_1], \ldots  , [v_k])&\mapsto [v_1 \otimes \cdots \otimes v_k]
			\end{align*}
	When dealing with complex projective spaces we will denote by $X_{n_1,\dots,n_k}=\nu(Y_{n_1,\dots,n_k})$ the Segre variety of the multiprojective space $Y_{n_1,\dots,n_k}=\mathbb{P}^{n_1}\times \cdots \times \mathbb{P}^{n_k}$. 
	\end{notation}
	We recall that the $r$-th \emph{secant variety} $\sigma_k(X_{n_1,\dots,n_k})$ of a Segre variety $X_{n_1,\dots,n_k}\subset \mathbb{P}^N$ is defined as
	$$
	\sigma_r(X_{n_1,\dots,n_k})=\overline{\{q\in \mathbb{P}^N \colon r(q)=r\}}.
	$$
The variety $X_{n_1,
	\dots ,n_k}$ is said to be \emph{r-defective} if $$\dim (\sigma_r(X_{n_1,\dots,n_k}))< \min \{ r(\dim X_{n_1,\dots,n_k}+1)-1,N  \}.$$

	 Since the algorithm we are going to present is based on the classification \cite[Theorem 7.1]{BBS}, we briefly recall it here in the revised version of our Theorem \ref{newMAINTHM}. 
	\smallskip

	\paragraph*{\textbf{The classification}}{\emph{\cite[Theorem 7.1 revised]{BBS}-Theorem \ref{newMAINTHM} in the present paper.}}\label{main_theorem} A concise rank-3 tensor $T\in \mathbb{C}^{n_1}\otimes \cdots \otimes \mathbb{C}^{n_k}  $ is identifiable except if $T$ is in one of the following families.
	 \begin{enumerate} [label=\alph*)]
	 	\item \label{teorcaso1} \textbf{[Matrix case]} The first trivial example of non-identifiable rank-3 tensors are $3 \times 3 $ matrices, which is a very classical case.
	 	\item\label{teorcaso2} \textbf{[Tangential case]} The \emph{tangential variety} of a variety is the tangent developable of the variety itself. A point $q$ essentially lying on the tangential variety of the Segre $X_{1,1,1} $ is actually a point of the tangent space $T_{[p]}X_{1,1,1}$ for some $p=u\otimes v \otimes w\in (\mathbb{C}^2)^{\otimes 3}$. Therefore there exists some $a,b,c\in \mathbb{C}^2 $ such that $T$ can be written as 
	 	\begin{align*}
	 		T=a\otimes v\otimes w+u\otimes b \otimes w+u\otimes v \otimes c
	 	\end{align*}
	 	and hence $q$ is actually non-identifiable.
	 	\item\label{teorcaso3} \textbf{[Defective case]}  We recall that the third secant variety of a Segre variety $X_{n_1,\dots,n_k}$ is defective if and only if $(n_1,\dots ,n_k)=(1,1,1,1),(1,1,a)$ with $a\geq 3$ (cf. \cite[Theorem 4.5]{AOP}). We will see that the latter case will not play a role in the discussion and hence we can focus on the case $k=4$. 
	 	By defectivity, the dimension of  $\sigma_3(X_{1,1,1,1})$ is strictly smaller than the expected dimension and this proves that the generic element of $\sigma_3(X_{1,1,1,1})$ has an infinite number of rank-3 decompositions and therefore all the rank-3 tensor of this variety have an infinite number of decompositions. 
	 	\item[d),e)]\textbf{[Conic cases]} In this case one works with the Segre variety $X_{2,1,1}$ given by the image of a projective plane and two projective lines. 
	 	
	 	Let $Y_{2,1,1}=\mathbb{P}^2\times \mathbb{P}^1\times \mathbb{P}^1$. Consider the Segre variety $X_{1,1} \subset \mathbb{P}^3$ given by the last two factors of $Y_{2,1,1}$ and take a hyperplane section which intersects $X_{1,1}$ in a conic $\mathcal{C }$. Let $L_{\mathcal{C}}$ be the Segre given by the product of the first factor $\mathbb{P}^2$ of $Y_{2,1,1}$ and the conic $\mathcal{C}$, therefore $L_{\mathcal{C}}\subset X_{2,1,1}$. The family of non-identifiable rank-3 tensors are points lying in the span of $L_{\mathcal{C}}$.
	 	In this case, the non-identifiability comes from the fact that the points on $\langle \mathcal{C} \rangle $ are not identifiable and the distinction between the two cases reflects the fact that the conic $\mathcal{C}$ can be either irreducible or reducible. The distinction between the two cases can be expressed as follows working in coordinates:
	 	\begin{enumerate}[label=\alph*),start=4]
	 		\item\label{teorcaso4} The non-identifiable tensor $T\in\mathbb{C}^3\otimes \mathbb{C}^2\otimes \mathbb{C}^2$ and there exists a basis $\{u_1,u_2,u_3\}\subset\mathbb{C}^3 $ and a basis $\{ v_1,v_2 \} \subset\mathbb{C}^2 $ such that $ T$ can be written as 
	 		\begin{align*}
	 			T= u_1 \otimes v_1^2+u_2\otimes v_2^2 + u_3 \otimes (\alpha v_1+\beta v_2)^2,
	 		\end{align*}
	 		for some $\alpha,\beta\neq 0 $;
	 		\item\label{teorcaso5}The non-identifiable tensor $T\in\mathbb{C}^3\otimes \mathbb{C}^2\otimes \mathbb{C}^2$ and there exists a basis  $\{u_1,u_2,u_3\}\subset\mathbb{C}^3 $ and a basis $\{ v_1,v_2 \} \subset \mathbb{C}^2 $ such that $ T$ can be written as  
	 		\begin{align*}
	 			T= u_1 \otimes v_1\otimes \tilde{p}+u_2\otimes v_2 \otimes \tilde{p}+ u_3\otimes \tilde{q} \otimes w,  
	 		\end{align*}
	 		for some $\tilde{q}\in \langle v_1,v_2 \rangle $, where $\tilde{p},w \in \mathbb{C}^2$ must be linearly independent;
	 		\end{enumerate}
 		\skipitems{2}
	 	\item\label{teorcaso6} \textbf{[General case]} The last family of non-identifiable rank-3 tensors relates the Segre variety $X_{n_1,n_2,1^{k-2}}$ that is the image of the multiprojective space $Y_{n_1,n_2,1^{k-2}}=\mathbb{P}^{n_1}\times \mathbb{P}^{n_2}\times (\mathbb{P}^1)^{(k-2)}$, where either $k\geq 4$ and $n_1,n_2\in \{1,2\}$ or $k=3$ and $(n_1,n_2,n_3)\neq (2,1,1)$.
	 	The non-identifiable rank-3 tensors of this case are as follows.
	 	Let $ Y':=\mathbb{P}^1\times \mathbb{P}^1\times \{ u_3\} \times \cdots \times \{ u_k\}$ be a proper subset of $Y_{n_1,n_2,1^{k-2}} $, take $q'$ in the span of the Segre image of $Y'$ with the constrain that $q'$ is not an elementary tensor. Therefore $q'$ is a non-identifiable tensor of rank-2 since it can be seen as a $2\times 2$ matrix of rank-2. Let $p\in X_{n_1,n_2,1^{k-2}}$ be a rank-1 tensor taken outside the Segre image of $Y'$. Now any point $q \in \langle  \{q',p \}  \rangle \setminus \{q' , p\}  $ is a rank-3 tensor (cf. Proposition \ref{prop:new statement}) and it is not identifiable since $q'$ has an infinite number of decompositions and each of these decompositions can be taken by considering $p$ together with a decomposition of $q'$.
\smallskip	 	

	 	 For a coordinate description of this case, we take $T \in \mathbb{C}^{m_1}\otimes \mathbb{C}^{m_2}\otimes (\mathbb{C}^2)^{k-2} $, where $k\geq 3$, $m_1,m_2\in \{2,3 \} $ such that $m_1+m_2+ (k-2)\geq 4$. Moreover there exist distinct $a_1,a_2\in \mathbb{C}^{m_1}$, distinct $b_1,b_2\in \mathbb{C}^{m_2}$ and for all $i\geq 3$ there exists a basis $ \{u_i,\tilde{u}_i\}$ of the $i$-th factor such that
	 	$ T$ can be written as
	 	\begin{align*}
	 		T=& (a_1\otimes b_1+a_2\otimes b_2)\otimes u_3 \otimes \cdots \otimes u_k + a_3 \otimes b_3\otimes \tilde{u}_3 \otimes \cdots \otimes \tilde{u}_k,
	 	\end{align*}    
	 	where if $m_1=2$ then $a_3\in \langle  a_1,a_2\rangle$ otherwise $ a_1,a_2,a_3$ are linearly independent. Similarly, if $ m_2=2$ then $ b_3 \in \langle b_1,b_2\rangle$, otherwise $ b_1,b_2,b_3$ form a basis of the second factor.
	 	\end{enumerate}

For a more detailed overview of the next couple of sections we refer to \cite{phdtesi}.

\subsection{Concision}\label{Subsec:Concision}

Fix a tensor $T\in \mathbb{C}^{n_1}\otimes \cdots \otimes \mathbb{C}^{n_k}$, where $ k\geq 2$ and $n_1,\dots,n_k\geq 1 $. For all $ \ell=1,\dots, k$, denote by $\mathcal{B}_\ell=\{e^\ell_1,\dots,e^\ell_{n_\ell}\} $ an ordered basis of $\mathbb{C}^{n_\ell} $ and by $\mathcal{B}^*_\ell=\{ \eta^\ell_1,\dots,\eta^\ell_{n_\ell} \} $ the corresponding dual basis. Let $T= (t_{i_1,i_2,\cdots, i_k})$ be the coordinates of $T$ with respect to those bases. 


A useful operation that allows to store the elements of a tensor as a matrix is the flattening (cf. \cite[Section 3.4]{Lands}), also called matrix-unfolding of a tensor in \cite[Definition 1]{flat}, which is the oldest reference we found for a formal definition of this operation.
\begin{definition}\label{flattening} The $\ell$-th \emph{flattening} of a tensor $ T\in \mathbb{C}^{n_1}\otimes \cdots \otimes \mathbb{C}^{n_k}$ whose coordinates in the canonical basis $\{e_{i_1}^1 \otimes \cdots \otimes e_{i_k}^k\}$ are $t_{i_1,\dots,i_k}$ is the linear map 
	\begin{align*}
		\varphi_\ell\colon &( \mathbb{C}^{n_1}\otimes \cdots \otimes \mathbb{C}^{n_{\ell-1}}\otimes \mathbb{C}^{n_{\ell+1}}\otimes \cdots \otimes \mathbb{C}^{n_k})^*  \rightarrow\mathbb{C}^{n_\ell}\\
		f&\mapsto \sum_{i_1,\dots,i_k} t_{i_1\dots i_k} f(e^1_{i_1}\otimes \cdots \otimes e^{\ell-1}_{i_{\ell-1 }}\otimes e^{\ell+1}_{i_{\ell+1}}\cdots \otimes e^{k}_{i_k}) e^\ell_{i_\ell} .
	\end{align*}
	We denote by $T_\ell $ the $ n_\ell \times (\prod_{i\neq \ell} n_i) $ associated matrix with respect to bases $\mathcal{B}_\ell $ and $\{\eta^1_1\otimes\cdots \otimes \eta^{\ell-1}_1\otimes \eta^{\ell+1}_1 \otimes \cdots \otimes \eta^k_1,\eta^1_1\otimes\cdots \otimes \eta^{\ell-1}_1\otimes \eta^{\ell+1}_1 \otimes \cdots \otimes \eta^k_2,\dots ,\eta^1_{n_1}\otimes\cdots \otimes \eta^{\ell-1}_{n_{\ell-1}}\otimes \eta^{\ell+1}_{n_{\ell+1}} \otimes \cdots \otimes \eta^k_{n_k}  \} $. 
\end{definition}

\begin{definition}[\protect{\cite{hit1928}}]
	Let $T\in \mathbb{C}^{n_1}\otimes \cdots \otimes \mathbb{C}^{n_k}$.
	For all $\ell=1,\dots,k$ let $T_\ell$ be the $\ell$-th flattening of $T$ as in Definition \ref{flattening} and denote by $r_\ell:=r(T_\ell) $. The \emph{multilinear rank} of $ T$ is the $k$-uple $$mr(T):=(r_1,\dots,r_k) $$ containing the ranks of all the flattenings of $T$.
\end{definition}

We remark that (cf. \cite[Theorem 7]{CK}) for all $ \ell=1,\dots,k$ 
\begin{equation}\label{mrank}
	r_\ell \leq r(T) \leq \prod_{i\neq \ell}r_i\end{equation}
and moreover it is classically known that
\begin{equation*}
	r(T)=1 \iff \textit{ the multilinear rank of }T \textit{ is } (1,\dots,1).  
\end{equation*} 

We are now ready to recall the concision process for a tensor. The following Lemma is the base step also for the algorithm we are going to construct in order to test the possible identifiability of a given tensor $T$. 
\begin{lemma}[\protect{Concision/Autarky, \cite[Prop. 3.1.3.1]{Lands}, \cite[Lemma 2.4]{BB3sec}}]\label{concision}
	For any $T \in \mathbb{C}^{n_1}\otimes \cdots \otimes \mathbb{C}^{n_k} $ one can uniquely determine minimal integers $ k'\leq k$ and $n'_1,\dots, n'_{k'}$ with $n'_i\leq n_i $ such that 
	\begin{itemize}
	\item $ T\in \mathbb{C}^{n'_1}\otimes \cdots \otimes \mathbb{C}^{n'_{k'}}\subseteq \mathbb{C}^{n_1}\otimes \cdots \otimes \mathbb{C}^{n_k} $;
	\item the rank of $T $ as an element of $\mathbb{C}^{n_1}\otimes \cdots \otimes \mathbb{C}^{n_k}  $ is the same as the rank of $ T$ as an element of $\mathbb{C}^{n'_1}\otimes \cdots \otimes \mathbb{C}^{n'_{k'}} $;
	\item any rank decomposition of $ T$ can be found in $\mathbb{C}^{n'_1}\otimes \cdots \otimes \mathbb{C}^{n'_{k'}} $.
	\end{itemize}
	We denote by $\mathcal{T}_{n'_1,\dots,n'_{k'}}:=\mathbb{C}^{n'_1}\otimes \cdots \otimes \mathbb{C}^{n'_{k'}} $ and we will call it the concise tensor space of $ T$.
\end{lemma}
The lemma states that for any tensor $T\in \mathbb{C}^{n_1}\otimes \cdots \otimes \mathbb{C}^{n_k} $ there exists a unique minimal tensor space included in $\mathbb{C}^{n_1}\otimes \cdots \otimes \mathbb{C}^{n_k} $ that contains both the tensor and all its possible rank decompositions. Let us review more in details a procedure that computes the concise tensor space $ \mathcal{T}_{n'_1,\dots,n'_{k'}}$ of a given tensor $ T\in \mathbb{C}^{n_1}\otimes \cdots \otimes \mathbb{C}^{n_k}$ working in coordinates.
\smallskip

After having fixed basis of $\mathbb{C}^{n_1}\otimes \cdots \otimes \mathbb{C}^{n_k}$, let $T=(t_{i_1,\dots, i_k})\in \mathbb{C}^{n_1}\otimes \cdots \otimes \mathbb{C}^{n_k} $ be its coordinate representation , where all $n_i\geq 1 $ and $k\geq 2 $.
For all $\ell=1,\dots,k $ consider the $\ell$-{th}  flattening $T_\ell$ of $ T$ as in Definition \ref{flattening}. 
For the sake of simplicity take $\ell=1 $.
The first column of $ T_1$ is 
\begin{align*}
	(t_{1,1,\dots ,1},t_{2,1,\dots ,1},\dots,t_{n_1,1,\dots ,1 } )^T=\sum_{i=1}^{n_1}t_{i,1,\dots ,1}u^1_i=\sum_{i,j=1}^{n_1}t_{i,1,\dots, 1}\alpha^1_j (u^1_i),
\end{align*}
which is referred to $u^2_1\otimes \cdots \otimes u^k_1 $. The same holds for the other columns of $T_1 $. 
Once we have computed $n'_1:=r(T_1)$ we can extract $n'_1$ linearly independent columns from $ T_1$, say $u^1_1,\dots, u^1_{n'_1} $. Since $\hbox{Im}(\varphi_1)= \langle u^1_1,\dots,u^1_{n'_1} \rangle\cong \mathbb{C}^{n'_1}\subseteq \mathbb{C}^{n_1} $, we rewrite the other columns as a linear combination of the independent ones.
The resulting tensor $ T'$  will therefore live in a smaller space $\mathbb{C}^{n'_1}\otimes \mathbb{C}^{n_2} \otimes \cdots \otimes \mathbb{C}^{n_k}  $. By continuing this process for each flattening we arrive to the concise tensor space $$\mathcal{T}_{n'_1,\dots,n'_{k'}}=\mathbb{C}^{n'_1}\otimes \cdots \otimes \mathbb{C}^{n'_{k'}} $$ where we may assume $ n'_i>1$ for all $i=1,\dots,k' $ and $ k'\leq k$ since $ \mathbb{C}^{n'_1}\otimes \cdots \otimes \mathbb{C}^{n'_{k'}}\otimes \{u_1\} \otimes \cdots \otimes \{ u_{k-k'}\}\cong \mathbb{C}^{n'_1}\otimes \cdots \otimes \mathbb{C}^{n'_{k'}}$.

\subsection{Matrix pencils}\label{Subsec:Pencil}
In this subsection we review some basic facts on matrix pencils that will be useful for the construction of the algorithm. We will briefly describe how to achieve the Kronecker normal form of any matrix pencil and we refer to \cite[Vol. 1, Ch. XII]{GantV12} for a detailed exposition. 

For the rest of this subsection, unless specified, we will work over an arbitrary field $\mathbb{K}$ of characteristic $0$.
\smallskip

Fix integers $m,n>0$. A \emph{polynomial matrix} $A(\lambda)$  is a matrix whose entries are polynomials in $\lambda$, namely
$$A(\lambda)=(a_{i,j}(\lambda))_{i=1,\dots, m, j=1,\dots, n}, \hbox{ where } a_{i,j}(\lambda):=a^{(0)}_{i,j}+a^{(1)}_{i,j}\lambda +\cdots + a^{(l)}_{i,j}\lambda^l,$$
for some $l>0$. 
If we set $A_k:=(a_{i,j}^{(k)})$, then we can write $A(\lambda)$ as
$$A(\lambda)=A_0+\lambda A_1 +\cdots +\lambda^l A_l .$$ The \emph{rank} $r(A(\lambda))$ of $A(\lambda)$ is the positive integer $r$ such that all $r+1$ minors of $A(\lambda)$ are identically zero as polynomials in $\lambda$ and there exists at least one minor of size $r$ which is not identically zero. A \emph{matrix pencil} is a polynomial matrix of type $A(\lambda)=A_0+\lambda A_1$. Given two matrix pencils $A(\lambda)=A_0+\lambda A_1$ and $B(\lambda)=B_0+\lambda B_1$, we say that $A(\lambda)$ and $B(\lambda )$ are \emph{strictly equivalent} if there exist two invertible matrices $P,Q$ such that
$$ P(A_0+ \lambda A_1)Q=B_0+\lambda B_1.$$
We shall see that the Kronecker normal form of a matrix pencil is determined by a complete system of invariants with respect to the strict equivalence relation defined above.

 Any matrix pencil $A+ \lambda B$ of size $m\times n$ can be either regular or singular:
\begin{definition}\label{pencil reg e sing}
Let $A,B\in M_{m,n}(\mathbb{K})$.
A pencil of matrices $A+\lambda B$ is called \emph{regular} if 
\begin{enumerate}
\item both $A$ and $B$ are square matrices of the same order $m$;
\item the determinant $\det(A+\lambda B)$ does not vanish identically in $\lambda$.
\end{enumerate}
Otherwise the matrix pencil is called \emph{singular}.

\end{definition}
We now recall how to find the normal form of a pencil $A+\lambda B$ depending on whether it is regular or not.
\subsubsection{Normal form of regular pencils}
In the case of regular pencils, normal forms can be found by looking at the elementary divisors of a given matrix pencil. In order to introduce them, it is convenient to consider the pencil $A+\lambda  B$ with homogeneous parameters $\lambda , \mu$, i.e. $\mu A +\lambda  B$.

 Let $\mu  A  + \lambda  B$ be the rank $r$ homogeneous matrix pencil associated to $A+\lambda B$. For all $j=1,\dots,r$, denote by $D(\lambda,\mu)_j$ the greatest common divisor of all the minors of order $j$ in $\mu A+ \lambda B$ and set $D_0(\lambda,\mu)=1$.
Define the following polynomials
$$i_j(\lambda,\mu):=\frac{D_{r-j+1}(\lambda,\mu)}{D_{r-j}(\lambda,\mu)}, \hbox{ for all } j=1,\dots,r. $$
Note that all $i_j(\lambda,\mu)\in \mathbb{K}[\lambda,\mu]$ can be splitted into products of powers of irreducible homogeneous polynomials that we call \emph{elementary divisors}.
Elementary divisors of the form $\mu^q$ for some $q>0$ are called \emph{infinite elementary divisors}.


One can prove that two regular pencils $A+ \lambda B$ and $A_1+\lambda B_1$ are strictly equivalent if and only if they have the same elementary divisors and infinite elementary divisors (cf. \cite[Vol. 2, Ch. XII, Theorem 2]{GantV12}).
Therefore elementary divisors and infinite elementary divisors are invariant with respect to the strict equivalence relation. Moreover they form a complete system of invariants for the strict equivalence relation since they are irreducible elements with respect to the fixed field $\mathbb{K}$. This is the reason why the polynomials $i_j(\lambda,\mu)$ defined above are actually called \emph{invariant polynomials} for all $ j=1,\dots,r$.

We recall that the \emph{companion matrix} of a monic polynomial $g(\lambda)=a_0+a_1\lambda +\cdots +a_{n-1}\lambda^{n-1}+\lambda^{n}$ is 
$$L=\begin{bmatrix}
	0& 1 & 0 &\hdots & 0 \\ \vdots & \ddots& \ddots &\ddots & \vdots \\ 0 & \hdots &0 & 1&0\\ 0 &\hdots & \hdots &0 &1 \\ -a_0 & -a_1 &\hdots & \hdots& -a_{n-1}
	\end{bmatrix}.
$$

\begin{theorem}[\protect{\cite[Vol. 2, Ch. XII, Theorem 3]{GantV12}}]
Every regular pencil $A+\lambda B$ can be reduced to a (strictly equivalent) canonical block diagonal form of the following type 
$$ [N^{(u_1)};\dots; N^{(u_s)};J_{v_1};\dots ; J_{v_t};L_{w_1};\dots ; L_{w_p}],$$
where

\begin{itemize}
\item The first $s$ diagonal blocks are related to infinite elementary divisors $\mu^{u_1},\dots,\mu^{u_s}$ of the pencil $A+\lambda B$
and for all $i=1,\dots, s$
$$N^{(u_i)}=\begin{bmatrix} 1 & \lambda & & \\ & \ddots&\ddots \\ & & 1& \lambda \\& & & 1 \end{bmatrix} \in M_{u_i}(\mathbb{K}).$$
\item The blocks $J_{v_i}$ are the Jordan blocks related to elementary divisors of type $(\lambda -\lambda_i)^{v_i}$.
\item The last p diagonal blocks $L_{w_1},\dots ,L_{w_p}$ are the companion matrices associated to the remaining elementary divisors of $A+\lambda B$. 
\end{itemize}
\end{theorem}

%
%

\subsubsection{Normal form of singular pencils}\label{sez pencil singolari}
In the previous case, a complete system of invariants was made by both elementary divisors and infinite ones. We shall see that, in case of singular pencils, this is not sufficient to determine a complete system of invariants with respect to the strict equivalence relation.
Fix $m\leq n$ and let $A+\lambda B$ be a singular pencil of rank $r$, where  $A,B\in M_{m,n}(\mathbb{K})$. Since the pencil is singular, the columns of $A+\lambda B$ are linearly dependent, therefore the system 
\begin{equation}\label{sist}
(A+\lambda B)x=0
\end{equation} has a non-zero solution with respect to $x$. Note that any solution $\tilde{x}$ of the above system is a vector whose entries are polynomials in $\lambda$, i.e. $\tilde{x}=\tilde{x}(\lambda)$. 
It has been proven in \cite[Vol. 2, Ch. XII, Theorem 4]{GantV12} that  if equation \eqref{sist} has a solution of minimal degree $\varepsilon \neq 0$ with respect to $\lambda$, the singular pencil $A+\lambda B$ is strictly equivalent to 
\begin{align*}
	\begin{bmatrix} L_{\varepsilon} & \\ & \hat{A}+\lambda \hat{B} \end{bmatrix},
\end{align*}
where 
\begin{align*}
	L_{\varepsilon}=\begin{bmatrix} \lambda & 1 &&\\ & \ddots & \ddots& \\  &  &\lambda &  1 \end{bmatrix} \in M_{\varepsilon,\varepsilon+1}(\mathbb{K}),
\end{align*}
and $\hat{A}+\lambda \hat{B}$ is a pencil of matrices for which the equation analogous to \eqref{sist} has no solution of degree less than $\varepsilon$.

By applying the previous result iteratively, a singular pencil $A+\lambda B$ is strictly equivalent to the block diagonal matrix $$[L_{\varepsilon_1} ;\dots;L_{\varepsilon_p};A_p+\lambda B_p],$$ where $0\neq \varepsilon_1\leq \cdots \leq \varepsilon_p$ and the last block is such that $(A_p+\lambda B_p)x=0$ has no non zero solution, i.e. the columns of $A_p+\lambda B_p$ are linearly independent. Then one looks at the rows of $A_p+\lambda B_p$. If these are linearly dependent, one can apply the same procedure just described by considering the associated system of the transposed pencil. 


Now let us treat the case in which there are some relations of degree zero (with respect to $\lambda$) between the rows and the columns of the given pencil $A+\lambda B$. Denote by $g$ and $h$ the maximal number of independent constant solutions of equations 
$$(A+\lambda B)x=0 \hbox{ and } (A^T+\lambda B^T)x=0 \hbox{ respectively}.$$ Let $e_1,\dots, e_g\in \mathbb{K}^n$ be linearly independent solutions of the system $(A+\lambda B)x=0$, completing them to a basis of $\mathbb{K}^n$ and rewriting the pencil with respect to this basis, we get $\tilde{A}+\lambda \tilde{B}= \begin{bmatrix} 0_{m\times g} & \tilde{A}_1+\lambda \tilde{B}_1\end{bmatrix}. $
One can do the same by taking $h$ linearly independent vectors that are solutions of the transpose pencil and hence the first $h$ rows of $\tilde{A}_1+\lambda \tilde{B}_1$ are zero with respect this new basis. Thus we obtain
$$ \begin{bmatrix}  0_{h\times g} & \\ & A_0+\lambda B_0  \end{bmatrix},$$
where $ A_0+\lambda B_0 $ does not have any degree zero relation, and hence either $A_0+\lambda B_0$ satisfies the assumptions of  \cite[Vol. 2, Ch. XII, Theorem 4]{GantV12}  or it is a regular pencil. 

There is a quicker way, due to Kronecker, to determine the canonical form of a given pencil, avoiding the iterative reduction just explained. It involves the notion of minimal indices. These last, together with elementary divisors (possibly infinite) will form a complete system of invariants for non singular pencils.\smallskip

Let $A+\lambda B$ be a non singular pencil and let $x_1(\lambda)$ be a non zero solution of least degree $\varepsilon_1$ for $(A+\lambda B)x=0$. Take $x_2(\lambda)$ as a solution of least degree $\varepsilon_2$ such that $x_2(\lambda)$ is linearly independent from $x_1(\lambda)$. Continuing this process, we get a so called \emph{fundamental series of solutions} of the system 
$$ x_1(\lambda),\dots, x_p(\lambda), \hbox{ of degrees } \varepsilon_1\leq \cdots \leq \varepsilon_p, \hbox{ for some } p\leq n.$$ 
We remark that a fundamental series of solution is not uniquely determined, but one can show that the degrees $\varepsilon_1,\dots, \varepsilon_p$ are the same for any fundamental series associated to a given system $(A+\lambda B)x=0$.
The \emph{minimal indices for the columns} of $A+\lambda B$ are the integers $\varepsilon_1,\dots,\varepsilon_p$. Similarly, the \emph{minimal indices for the rows} are the degrees $\eta_1,\dots, \eta_q$ of a fundamental series of solutions of $(A^T+\lambda B^T)x=0$.
Strictly equivalent pencils have the same minimal indices (cf. \cite[Vol. 2, Ch. XII, Sec. 5, Par. 2]{GantV12}).

Now let $A+\lambda B$ be a singular pencil and consider its normal form
\begin{equation}\label{forma pencil singolare}
\begin{bmatrix}
0_{h\times g}  & & & & & & &\\
 &L_{\varepsilon_{g+1}} & & & & & &\\
 & &\ddots & & & & &\\
 & & & L_{\varepsilon_{p}} & & & &\\
 & & & &L^T_{\eta_{h+1}} & & &\\
 & & & & & \ddots & &\\
 & & & & & & L^T_{\eta_{q}}&\\
 & & & & & & & A_0+\lambda B_0\\
\end{bmatrix}.\end{equation}
\begin{remark}
The system of indices for the columns (rows) of the above block diagonal matrix is obtained by taking the union of the corresponding system of minimal indices of the individual blocks.
\end{remark}
 We want to determine minimal indices for the above normal form \eqref{forma pencil singolare}.
By the previous remark, it is sufficient to determine the minimal indices for each block. Clearly the regular block $A_0+\lambda B_0$ has no minimal indices, the zero block $0_{h\times g}$ has $g$ minimal indices for columns and $h$ minimal indices for rows all equal to zero respectively, namely $\varepsilon_1=\cdots =\varepsilon_g=\eta_1=\cdots =\eta_h=0$.  The block $L_{\varepsilon_i}\in M_{\varepsilon_i,\varepsilon_i +1}(\mathbb{K})$ has linearly independent rows, therefore it has just one minimal index for column $\varepsilon_i$ for all $i=1,\dots, p$. Similarly, for all $j=1,\dots,q $ the block $L_{\eta_j}$ has just one minimal index for rows $\eta_{j}$.

We conclude that the canonical form \eqref{forma pencil singolare} is completely determined by both the minimal indices $\varepsilon_1,\dots,\varepsilon_p,$ $\eta_1,\dots,\eta_q$ and the elementary divisors. 

It is classically attributed to Kronecker the result proving that two arbitrary pencils $A+\lambda B$ and $A_1+\lambda B_1$ of rectangular matrices are strictly equivalent if and only if they have the same minimal indices and the same elementary divisors (possibly infinite).
We conclude this part by illustrating with an example how to construct the Kronecker normal form of a matrix pencil.

\begin{example}

Consider the pencil
$$
A+\lambda B=
\begin{bmatrix}
1 & 0&\lambda & 3\lambda +1 & 1 & 2\\
2\lambda & \lambda & \lambda & 3 & \lambda &0\\
0 &0&0& 1&1&1\\
2\lambda +1& \lambda & 2\lambda +1& 3\lambda +4 & \lambda +1 & 2
\end{bmatrix}.
$$
The kernel of the system $(A+\lambda B)x=0$ is generated by
$$ \mathrm{Ker} (A+\lambda B)= \langle  \begin{bmatrix} 1 \\1\\-3\\1\\0\\-1 \end{bmatrix}, \begin{bmatrix} 1\\-3\\0\\0\\1\\-1 \end{bmatrix} , \begin{bmatrix} -\lambda^2\\ 2\lambda^2-\lambda-1\\\lambda \\0\\ 0\\ 0 \end{bmatrix} \rangle.$$
Since the minimum integer of the non constant solution is $\varepsilon =2$, we know that the normal form of the pencil contains the following block $$L_{2}=\begin{bmatrix} 
\lambda & 1 &0 \\
0 & \lambda &1
\end{bmatrix}.$$
Moreover, we see that there are $g=2$ linearly independent constant solutions. Considering the transpose pencil, then $$\mathrm{Ker} ((A+\lambda B)^T)=\langle \begin{bmatrix} -1\\-1\\0\\1\end{bmatrix}\rangle,$$ so there is just one constant solution. Therefore, keeping the above notation, $\eta=0$ and $h=1$. Moreover the invariant polynomials of the pencil are $i_4(\lambda,\mu)=0$, $i_3(\lambda,\mu)=\mu$ and all the others are equal to $1$. Therefore the Kronecker normal form of $A+\lambda B$ is 
$$\left[
\begin{array}{c|c|c}
\begin{matrix}
0&0
\end{matrix} & & \\
\hline 
& \begin{matrix}
\lambda  &1 & 0 \\ 0 & \lambda & 1
\end{matrix} &\\
\hline
& & 1
\end{array}
\right].
$$
\end{example}

\subsubsection{3-factors tensor spaces and matrix pencils }\label{penciletensori}
From now on we work again over $\mathbb{C}$.
Any tensor $T\in \mathbb{C}^2\otimes \mathbb{C}^m\otimes \mathbb{C}^n$ can be seen as a matrix pencil via the isomorphism
\begin{align*}
\mathbb{C}^2\otimes (\mathbb{C}^m)^*\otimes (\mathbb{C}^n)^* &\xrightarrow{\sim} \{ \mathbb{C}^m\times \mathbb{C}^n\xrightarrow{\Phi} \mathbb{C}^2\} .
\end{align*}
We can easily pass from a tensor $T\in \mathbb{C}^2\otimes \mathbb{C}^m\otimes \mathbb{C}^n$ to its associated matrix pencil (and viceversa) by fixing a basis on each factor and looking at $ T$ in its coordinates with respect to the fixed bases. 
For example, let us fix the canonical basis on each factor and let $T=(t_{ijk})\in \mathbb{C}^2\otimes \mathbb{C}^m\otimes \mathbb{C}^n$. We can associate to $T$ the map 
\begin{align*}
\Phi_{T}:\mathbb{C}^m\times \mathbb{C}^n &\longrightarrow \mathbb{C}^2\\
(v,w) & \mapsto (v^TAw,v^TBw)
\end{align*}
where
$$A=(t_{1ij})_{i=1,\dots,m,j=1,\dots,n} \hbox{ and } B=(t_{2ij})_{i=1,\dots,m,j=1,\dots,n}. $$

Fixing the integer $m$ equal to either $2$ or $3$ in $\mathbb{C}^2\otimes \mathbb{C}^m\otimes \mathbb{C}^n$ leads us to consider very special tensor formats, namely $\mathbb{C}^2\otimes \mathbb{C}^2\otimes \mathbb{C}^n $ and $ \mathbb{C}^2\otimes \mathbb{C}^3\otimes \mathbb{C}^n$. In these cases there is a finite number of orbits with respect to the action of products of general linear groups (cf. \cite{K}). Such cases have been widely studied in \cite{Par}, where the author gave a complete orbit classification working in the affine setting. 

Remark that for any tensor belonging to either $\mathbb{C}^2\otimes \mathbb{C}^2\otimes \mathbb{C}^n$ or $\mathbb{C}^2\otimes \mathbb{C}^3\otimes \mathbb{C}^n$ one can consider the associated matrix pencil and, by computing its Kronecker normal form, it is possible to understand its rank. This last result comes from the following more general statement that is historically attributed to Grigoriev, J\'{a}J\'{a} and Teichert. We refer to \cite[Remark 5.4]{BL} for a historical note on the theorem.
\begin{theorem}[\protect{\cite{grigoriev}, \cite{jaja}, \cite{teichert}}]\label{jajaeco}
Let $T\in \mathbb{C}^2\otimes \mathbb{C}^m\otimes \mathbb{C}^n$ and let $A$ be the corresponding pencil with minimal indices $\varepsilon_1,\dots,\varepsilon_p,\eta_1,\dots,\eta_q$ and regular part $C=A_0+\lambda B_0$ of size $N$. Let $\delta(C)$ be the number of non-squarefree invariant polynomials of $C$. Then $T$ is a tensor of rank
\begin{align}\label{eqrangopencil}
 \sum_{i=1}^p (\varepsilon_i +1) + \sum_{i=1}^q (\eta_j +1) + N + \delta(C). 
\end{align}

\end{theorem}

In \cite{BL} the authors reviewed the orbits classification made in \cite{Par} and gave a geometric interpretation of the projectivization of all the orbits closures appearing in both cases. In the following section we will refer to the classification of \cite{BL} when necessary.


\section{Algorithm for the non-identifiability of a rank-3 tensor}\label{Sec:Algoritmo}
The purpose of this section is to write Algorithm \ref{algoalgo} where we can determine if a rank-3 tensor is not identifiable.

 All possible cases of non-identifiabile rank-$ 3$ tensors are collected in Theorem \ref{newMAINTHM}. 
\begin{itemize}
	\item The \textbf{input} of the algorithm we propose is a tensor $ T=(t_{i_1,i_2,\cdots ,i_k})\in \mathbb{C}^{n_1}\otimes\cdots \otimes \mathbb{C}^{n_k} $ presented in its coordinate descripiton with respect to canonical basis, where $k\geq 3 $, all $ n_j\geq 1$ and all $ i_j=1,\dots, n_j$, $j=1,\dots,k$. 
\item The \textbf{output} of the algorithm is a statement telling if the given tensor is a rank-3 tensor that falls into one of the cases mentioned above or not. 
\end{itemize}
The first step of Algorithm \ref{algoalgo} is to compute the \emph{concise tensor space} $ \mathcal{T}_{n'_1,\dots,n'_{k'}}=\mathbb{C}^{n'_1}\otimes \cdots \otimes \mathbb{C}^{n'_{k'}}$ \emph{of} $ T$ that we have already detailed in Subsection \ref{Subsec:Concision}, hence from now on we will work with concise tensors. Based on the resulting concise tensor space $\mathcal{T}_{n'_1,\dots,n'_{k'}}$, we split the algorithm into two different parts depending on whether $\mathcal{T}_{n'_1,\dots,n'_{k'}}$ is made by three factors or not. Subsection \ref{Subsec:3fattori} is devoted to the $3$-factors case while we refer to Subsection \ref{Subsec:+3fattori} for the other case.

\smallskip

\begin{remark}\label{spconciso}
Fix a tensor $T\in \mathbb{C}^{n_1}\otimes \cdots \otimes \mathbb{C}^{n_k}$ and compute the multilinear rank of $T$. By using the left inequality in \eqref{mrank} on each flattening $\varphi_\ell $, we are able to exclude some of the cases in which $r(T)$ is higher than 3. In those cases the algorithm stops since we are interested in  rank-$3 $ tensors.
Moreover, if the multilinear rank of $T$ contains more than $k-3$ positions equal to $1$ then $T$ is either a rank-1  tensor or a matrix and we can also exclude these cases. Lastly, we remark that since the concise Segre of a rank-3 tensor is $\nu(\mathbb{P}^{m_1}\times \cdots \times \mathbb{P}^{m_{k}})$ where all $m_i\in \{1,2\}$ for all $i=1,\dots,k$, if one of the values in $mr(T)=(\dim(\mathbb{C}^{m_i+1}))_{i=1,\dots,k} $ is different from either $2$ or $3$ then we can immediately stop the algorithm.
Therefore, at the end of the concision process, we deal only with a tensor $T'\in \mathbb{C}_1^{n'_1}\otimes \cdots \otimes \mathbb{C}_{k'}^{n'_{k'}} $ such that
\begin{itemize}
 \item $r(T')\geq 2 $, 
\item $ 3\leq k'\leq k$ 
\item all $n'_i\in \{ 2,3\} $. 
\end{itemize}

\end{remark}

Now, depending on whether $ k'=3$ or $k'\geq 4$, we split the algorithm in two different parts.

\subsection{Three factors case}\label{Subsec:3fattori}
This subsection is devoted to treat the case in which the concise tensor space of the tensor $ T$ given in input has three factors.
By Remark \ref{spconciso}, the concise space $\mathcal{T}_{n_1,\dots,n_{k}}=\mathbb{C}^{n_1}\otimes \cdots \otimes \mathbb{C}^{n_{k}}$ of a tensor $T$ is such that all $n_i\in \{2,3\}$. Moreover, if $k=3 $ the only possibilities for $\mathcal{T}_{n_1,n_2,n_3}$, up to a reordering of the factors, are: 
\begin{itemize}
\item \label{111} $\mathcal{T}_{2,2,2}=\mathbb{C}^2\otimes \mathbb{C}^2\otimes \mathbb{C}^2 $;
\item \label{211} $\mathcal{T}_{3,2,2}=\mathbb{C}^3\otimes \mathbb{C}^2\otimes \mathbb{C}^2 $;
\item \label{221} $\mathcal{T}_{3,3,2}=\mathbb{C}^3\otimes \mathbb{C}^3\otimes \mathbb{C}^2 $;
\item \label{222} $\mathcal{T}_{3,3,3}=\mathbb{C}^3\otimes \mathbb{C}^3\otimes \mathbb{C}^3 $.
\end{itemize}

\begin{remark}\label{rem1}
The presence of a $\mathbb{C}^2$ in $\mathcal{T}_{2,2,2},\mathcal{T}_{3,2,2},\mathcal{T}_{3,3,2}$ allows to see all their elements  as a matrix pencil (cf. Subsection \ref{Subsec:Pencil}), in these cases we are also able to compute the rank of one of those tensors by classifying their at its associated matrix pencils (cf. Theorem \ref{jajaeco}).
\end{remark}
All the considerations made in the following will be summed up in Algorithm \ref{part2algo}  at the end of the subsection to which Algorithm \ref{algoalgo} will refer for the case of 3-factors.
\bigskip

\subsubsection{ $\mathcal{T}_{2,2,2}=\mathbb{C}^2\otimes \mathbb{C}^2\otimes \mathbb{C}^2$}\phantom{a}\\

\noindent
The second secant variety of $ X_{1,1,1}=\nu(\mathbb{P}^1\times \mathbb{P}^1\times \mathbb{P}^1)\subset \mathbb{P}^7$ fills the ambient space, i.e. $\dim \sigma_2(X_{1,1,1})=7 $. Consequently, any tensor $[T] \in \mathbb{P}^7\setminus X_{1,1,1} $ is either an element of the open part $\sigma_2^0(X_{1,1,1}) $ or an element of the tangential variety $ \tau(X_{1,1,1})$ of $ X_{1,1,1}$. Therefore if the concise tensor space of $T$ is $\mathcal{T}_{2,2,2}=\mathbb{C}^2\otimes \mathbb{C}^2\otimes \mathbb{C}^2$, rank-1 is excluded and $T$ has rank either 2 or 3. To detect the rank of $ T$ one can use the Cayley's hyperdeterminant which is the defining equation of $\tau(X_{1,1,1})$ (cf. \cite{GKZ}). Hence, if $T$ is a concise tensor in $\mathcal{T}_{2,2,2}$ and satisfies the hyperdeterminant equation, then $T$ has rank $3$ and it is not identifiable, otherwise it has rank 2.

\bigskip

\subsubsection{$ \mathcal{T}_{3,2,2}=\mathbb{C}^3\otimes \mathbb{C}^2\otimes \mathbb{C}^2$ }\label{subsub322}\phantom{a}\\

\noindent
The non-identifiable rank-3 tensors of $\mathcal{T}_{3,2,2}=\mathbb{C}^3\otimes \mathbb{C}^2\otimes \mathbb{C}^2$ come from cases \ref{3.} and \ref{4.} of Theorem \ref{newMAINTHM}.

If $\mathcal{T}_{3,2,2}$ is the concise tensor space of $T$, then obviously $r(T)\geq  3 $. Moreover, by \cite[Theorem 3.1.1.1]{Lands}, one can show that actually $ r(T)=3$ (cf. also \cite[Table 1]{BL}). 
Therefore every concise $ T\in \mathcal{T}_{3,2,2}$ is a rank-3 tensor. 
Moreover, since the dimension of the third secant variety of $X_{2,1,1}=\nu(\mathbb{P}^2\times \mathbb{P}^1\times \mathbb{P}^1)\subset \mathbb{P}^{11} $ is $\min\{ 14,11\} $, the generic fiber of the projection from the abstract secant variety  $\mathrm{Ab}\sigma_3(X_{1,1,1}):=\overline{\{ ((p_1,p_2,p_3),q) \in X_{1,1,1}^3\times \mathbb{P}^7 \colon q\in \langle p_1,p_2,p_3 \rangle  \}}$ to the secant variety has projective dimension 2, so the generic element of $ \sigma_3(X_{2,1,1})$ has an infinite number of decompositions. Therefore, by \cite[Cap II, Ex 3.22, part (b)]{Hart}, any rank-3 tensor in $\sigma_3(X_{2,1,1})$ is not identifiable, from which follows that any tensor whose concise tensor space is $\mathcal{T}_{3,2,2}=\mathbb{C}^3\otimes \mathbb{C}^2\otimes \mathbb{C}^2 $ is a non-identifiable rank-3 tensor. 
\begin{remark}
Rank-3 tensors can also live in $\sigma_2(X_{2,1,1})$ but a concise rank-3 tensor $T\in \mathcal{T}_{3,2,2}$ lies only on the third secant variety of $X_{2,1,1}$. 
\end{remark}

Both cases \ref{3.} and \ref{4.} of Theorem \ref{newMAINTHM} can be treated by looking at the matrix pencil associated to the corresponding tensor. 

\begin{remark}\label{rb}
In order to be consistent with the matrix pencil notation used in Subsection \ref{Subsec:Pencil} in which the first factor is used as a parameter space for the pencil, we swap the first and third factor of $ \mathcal{T}_{3,2,2}$, working now on $\mathcal{T}_{2,2,3}= \mathbb{C}^2\otimes \mathbb{C}^2\otimes \mathbb{C}^3$.
\end{remark}

\cite[Table 1]{BL} offers a complete description of all orbits in $ \mathbb{C}^2\otimes \mathbb{C}^2\otimes \mathbb{C}^3$, providing also the orbit closure in each case together with the Kronecker normal form of each orbit representative and its rank. Since we are working with concise rank-$3$ tensors of $\mathcal{T}_{2,2,3}$, we are interested in cases $ 7$ and $ 8$ of \cite[Table 1]{BL}, i.e.

\begin{table}[H]
{\renewcommand{\arraystretch}{1.2}
 \begin{tabular}{ | l | l | p{4cm} |}
	\hline
		Matrix pencil  &Tensor representative    &   \\ \hline
 $ \begin{bmatrix} \lambda & \mu & 0\\ 0 & 0 & \lambda\\ \end{bmatrix}$ & $a_1\otimes (b_1 \otimes c_1 + b_2 \otimes c_3) + a_2 \otimes b_1 \otimes c_2$  & case $7$ of \cite{BL}, \cite[Ex. 3.7]{BBS}, case \ref{4.}  \\ \hline
	 $\begin{bmatrix}\lambda & \mu & 0\\ 0 & \lambda & \mu\\ \end{bmatrix}$ & $a_1\otimes b_1 \otimes c_1 + a_2\otimes b_1 \otimes c_2 + a_1 \otimes b_2 \otimes c_2+a_2 \otimes b_2 \otimes c_3$  & case $8$ of \cite{BL}, \cite[Ex. 3.6]{BBS}, case \ref{3.}  \\ \hline
\end{tabular}
}
 \caption{Concise rank-3 tensors in $\mathbb{C}^2\otimes \mathbb{C}^2\otimes \mathbb{C}^3$}
\end{table}

where we considered all $ a_i$, $b_j$, $c_k$ are linearly independent elements of the corresponding factors and $\lambda, \mu $ represent homogeneous coordinates with respect to the first factor of $\mathcal{T}_{2,2,3}$.
Let us see which is the relation between the above Kronecker normal forms and our examples of non-identifiable rank-3 tensors in $\mathcal{T}_{2,2,3}$.

\begin{lemma}\label{lem1} The matrix pencil associated to any tensor $T\in \mathbb{C}^2\otimes \mathbb{C}^2\otimes \mathbb{C}^3$ belonging to \ref{4.} is of the following form:
\begin{align*}
\begin{bmatrix}
\lambda & \mu & 0\\
0 & 0 & \lambda
\end{bmatrix} \sim \begin{bmatrix}
\lambda & \mu & 0\\
0 & 0 & \mu
\end{bmatrix}.
\end{align*}
\end{lemma}
\begin{proof}
Let $ T\in \mathbb{C}^2\otimes \mathbb{C}^2\otimes \mathbb{C}^3$ be as in  case \ref{4.}, so 
$$T=\tilde{p}\otimes v_1\otimes u_1 +\tilde{p} \otimes v_2 \otimes u_2 +  w\otimes  (\alpha v_1+\beta v_2 ) \otimes u_3. $$
The matrix pencil associated to $T$ with homogeneous parameters $\lambda, \mu$ referred to the basis $\{ \tilde{p},w\}\subset\mathbb{C}^2$ is $$A= \begin{bmatrix}
\lambda & 0 & \alpha \mu\\
0 & \lambda & \beta \mu
\end{bmatrix}.$$
Since $ A$ is a singular pencil (cf. Definition \ref{pencil reg e sing}), in order to achieve the normal form of $ A$, we have to look at the minimum degree $ \varepsilon $ of the elements in $$ \hbox{Ker}(A)=\langle \begin{bmatrix}
-\alpha \mu \\ -\beta \mu \\ \lambda
\end{bmatrix} \rangle $$ with respect to $\lambda,\mu$ (cf. Subsection \ref{Subsec:Pencil}). Since $ \varepsilon=1$,  the normal form of $ A$ should contain a block of size $\varepsilon \times (\varepsilon+1) $ of this type $$ \begin{bmatrix}
\lambda & \mu & &   \\
& & \ddots & \ddots &   \\
& & &  \lambda &\mu  
\end{bmatrix}.
$$
Therefore we can conclude that \begin{equation*} A=\begin{bmatrix}
\lambda & \mu & 0\\
0 & 0 & \lambda
\end{bmatrix}.  \qedhere
\end{equation*}\end{proof}

\begin{corollary}\label{diffcases_pencil=} Let $ T\in \mathbb{C}^2\otimes \mathbb{C}^2\otimes \mathbb{C}^3$. The tensor $ T$ is a non-identifiable rank $ 3$ tensor coming from case \ref{4.} of Theorem \ref{newMAINTHM} if and only if the pencil associated to $ T$ is of the form 
$$ \begin{bmatrix}
\lambda & \mu & 0\\
0 & 0& \lambda
\end{bmatrix} \sim \begin{bmatrix}
\lambda & \mu & 0\\
0 & 0 & \mu
\end{bmatrix}.
$$ 
\end{corollary}

\begin{proof}
By Lemma \ref{lem1}, the matrix pencil associated to any tensor that belongs to case \ref{4.}
is $$ \hbox{ either } \begin{bmatrix}
\lambda & \mu & 0\\
0 & 0 & \lambda
\end{bmatrix} \hbox{ or }  \begin{bmatrix}
\lambda & \mu & 0\\
0 & 0 & \mu
\end{bmatrix}. $$
The viceversa also holds since actually the left above pencil corresponds to the tensor
$$ a_1\otimes b_1\otimes c_1 +a_1\otimes b_2 \otimes c_3 + a_2\otimes b_1 \otimes c_2 $$ (considering the first factor as a parameter space for the pencil) which is as in case \ref{4.}.
\end{proof}

\begin{lemma}\label{lemmac4} The matrix pencil associated to a tensor $T\in \mathbb{C}^2\otimes \mathbb{C}^2\otimes \mathbb{C}^3$ that is as in case \ref{3.} is
$$\begin{bmatrix}
\lambda & \mu & 0\\
0 & \lambda & \mu
\end{bmatrix}. $$
\end{lemma}
\begin{proof}
Let $T\in \mathbb{C}^2\otimes \mathbb{C}^2\otimes \mathbb{C}^3$ be as in case \ref{3.}, i.e. there is a basis $\{u_i\}_{i\leq 3}\subset \mathbb{C}^3$ and a basis $\{v_1,v_2\}\subset \mathbb{C}^2$ such that 
$$ T= v_1\otimes v_1\otimes u_1+ v_2\otimes v_2\otimes u_2+  (\alpha v_1+\beta v_2) \otimes(\alpha v_1+\beta v_2)\otimes u_3,$$
for some $(\alpha,\beta)\in \mathbb{C}^2\setminus \{ 0\} $. The matrix pencil associated to $T$ with homogeneous parameters $\lambda, \mu$ referred to the basis $\{ v_1,v_2\}\subset\mathbb{C}^2$ is
$$ A= \begin{bmatrix}
\lambda & 0& \alpha^2\lambda + \alpha \beta \mu \\
0 & \mu & \alpha \beta\lambda  + \beta^2\mu
\end{bmatrix}.$$
The kernel of $A$ is 
$$
\hbox{Ker}(A)=\langle \begin{bmatrix}
\alpha^2 \lambda \mu +\alpha \beta \mu^2\\ \alpha \beta \lambda^2+\beta^2\lambda \mu \\ -\lambda \mu
\end{bmatrix} \rangle,
$$
so the minimum degree $\varepsilon$ of the elements in $\hbox{Ker}(A)$ with respect to $\lambda,\mu$ is $2$.  
Therefore, the normal form of $A$ is 
\begin{align*}
\begin{bmatrix}
\lambda & \mu &0 \\
0 & \lambda & \mu\end{bmatrix}. &\qedhere
\end{align*}
 \end{proof}

\begin{corollary}\label{cor:caso8bl} Let $ T\in \mathbb{C}^2\otimes \mathbb{C}^2\otimes \mathbb{C}^3$. The tensor $ T$ is a non-identifiable rank-$ 3$ tensor coming from case \ref{3.} of Theorem \ref{newMAINTHM} if and only if the pencil associated to $ T$ is of the form 
$$ \begin{bmatrix}
\lambda & \mu & 0\\
0 & \lambda & \mu
\end{bmatrix}.
$$ \end{corollary}

\begin{proof}
By Lemma \ref{lemmac4}, the matrix pencil associated to any tensor that belongs to case \ref{3.} 
is $$\begin{bmatrix}
\lambda & \mu & 0\\
0 &  \lambda &\mu
\end{bmatrix}. $$
The viceversa also holds since actually the above pencil corresponds to the tensor
$$ e_1\otimes e_1 \otimes e_1 + (e_1\otimes e_2+e_2\otimes e_1)\otimes e_2 + e_2\otimes e_2\otimes e_3$$ which is as in case \ref{3.}.  
\end{proof}

\bigskip

\noindent
\subsubsection{$\mathcal{T}_{3,3,2}=\mathbb{C}^3\otimes \mathbb{C}^3\otimes \mathbb{C}^2 $ }\phantom{a}\\

\noindent

Let $\mathcal{T}_{3,3,2}$ be the concise tensor space of the tensor $T$ we have in input. We recall that the only non-identifiable rank-3 tensors in this case are the ones of case \ref{5.} of Theorem \ref{newMAINTHM} (cf. also Proposition \ref{prop:new statement}).
More precisely, let $ Y'=\mathbb{P}^1\times \mathbb{P}^1\times \{ w \}\subset Y_{2,2,1}=\mathbb{P}^2\times \mathbb{P}^2\times \mathbb{P}^1$. Take $q'\in \langle \nu(Y') \rangle \setminus \nu(Y_{2,2,1}) $ and $ p\in Y_{2,2,1}\setminus Y'$. Then $[T]\in \langle  q', \nu(p)\rangle $ is a rank-3 tensor and it is not identifiable.
If we take $\{u_i\}_{i\leq 3} \subset \mathbb{C}^3 $ as a basis of the first factor, $\{v_i\}_{i\leq 3}\subset  \mathbb{C}^3  $ as a basis of the second factor and $\{w,\tilde{w} \}\subset  \mathbb{C}^2$ as a basis of the third factor, then $ T$ is of the form
\begin{align}\label{tc6}
T=& u_1\otimes v_1\otimes w+u_2\otimes v_2 \otimes w+u_3 \otimes v_3\otimes \tilde{w}.
\end{align}  

Again we can look at this case by considering the associated matrix pencil of $T$. As before (cf. Remark \ref{rb}), to be consistent with the matrix pencil notation we already introduced, we swap the first and third factor of $\mathcal{T}_{3,3,2}$, working now on $\mathcal{T}_{2,3,3}=\mathbb{C}^2\otimes \mathbb{C}^3\otimes \mathbb{C}^3$. 

In \cite[Table 3]{BL}  are collected all Kronecker normal forms contained in $\mathcal{T}_{2,3,3}$. Since we are interested in rank-$ 3$ tensors having $ \mathcal{T}_{2,3,3}$ as concise tensor space, the only possibilities in terms of matrix pencils are 
\begin{align}\label{mp}
\begin{bmatrix}
\lambda & 0 & 0\\
0 & \lambda & 0 \\
0 & 0& \mu 
\end{bmatrix} \hbox{ and } \begin{bmatrix}
\lambda & 0 & 0\\
0 & \lambda+\mu & 0 \\
0 & 0& \mu 
\end{bmatrix}.
\end{align}
\begin{remark}
The matrix pencil associated to \eqref{tc6} is the first one in \eqref{mp} and it is easy to check that the tensor corresponding to the first matrix pencil in \eqref{mp} is actually $T$. 
\end{remark}
Therefore, if the concise tensor space of $T$ is $\mathcal{T}_{2,3,3,}$, it is sufficient to consider the normal form of the concise tensor $T'$ related to $ T$ and check if it corresponds to $$ \begin{bmatrix}
\lambda & 0 & 0\\
0 & \lambda & 0 \\
0 & 0& \mu 
\end{bmatrix}.$$

 Moreover, as in the previous case, we are able to detect the rank of any tensor having $\mathcal{T}_{2,3,3}$ as a concise tensor space (cf. Remark \ref{rem1}).

\bigskip

\noindent
\subsubsection{$\mathcal{T}_{3,3,3}=\mathbb{C}^3\otimes \mathbb{C}^3\otimes \mathbb{C}^3$}\phantom{a}\\

By Theorem \ref{newMAINTHM}, all rank-$ 3$ tensors whose concise tensor space is $\mathcal{T}_{3,3,3}$ are identifiable. Therefore if the concise tensor space of $T$ is  $\mathcal{T}_{3,3,3}$  we can immediately say that $ T$ does not belong to one of the $6$ families of non-identifiable rank-3 tensors.

We collect all the considerations made in this subsection in the following pseudo-algorithm.

\begin{algorithm}[H]
	\caption{(Three factors case)}\label{part2algo} 
	\noindent 
	\begin{flushleft}
\textbf{Input:} Concise tensor $ T=(t_{i_1,i_2,i_3}) \in \mathbb{C}^{n_1} \otimes \mathbb{C}^{n_2} \otimes \mathbb{C}^{n_3} $, with $n_i=2,3$ for all $i=1,2,3$ and $n_1\leq n_2\leq n_3$.\\
\textbf{Output:} A statement on whether $ T$ belongs to one of the six cases of non-identifiable rank-3 tensors or not.
\end{flushleft}
\begin{enumerate}
\item Case $(n_1,n_2,n_3)=(2,2,2) $.\\ If $ T$ satisfies Cayley's hyperdeterminant equation

\begin{align*}
\hbox{Hdet}(T):=\left( \begin{vmatrix}
t_{0,0,0} & t_{0,0,1} \\
t_{1,0,0} & t_{1,1,1}
\end{vmatrix} +\begin{vmatrix}
t_{0,1,0} & t_{0,0,1} \\
t_{1,1,0} & t_{1,0,1}
\end{vmatrix} \right)^2 -4 \begin{vmatrix}
t_{0,0,0} & t_{0,0,1} \\
t_{1,0,0} & t_{1,0,1}
\end{vmatrix} \cdot \begin{vmatrix}
t_{0,1,0} & t_{0,1,1} \\
t_{1,1,0} & t_{1,1,1}
\end{vmatrix}
\end{align*}
 the output is \emph{T belongs to case \ref{2.} of Theorem \ref{newMAINTHM} therefore it is not identifiable}. \\Otherwise the output is \emph{$ T$ is an identifiable rank-2 tensor}. 
\item Case $(n_1,n_2,n_3)=(2,2,3) $ (Remark that we already know that $T$ is not identifiable (cf. Subsection \ref{subsub322}), so we only need to classify it). \\
 Compute the Kronecker normal form of $ T$. 
\begin{itemize}
\item If the Kronecker normal form of $ T$ is $$\begin{bmatrix}
\lambda & \mu & 0\\
0 & 0& \mu
\end{bmatrix} $$ then the output is \emph{$T$ belongs to case \ref{4.} of Theorem \ref{newMAINTHM}, therefore it is not identifiable.} 
\item Else, $ T$ is as in case \ref{3.} and the output is \emph{$T$ belongs to case \ref{3.} of Theorem \ref{newMAINTHM} and it is not identifiable}. 
\end{itemize}
\item Case $(n_1,n_2,n_3)=(2,3,3) $. \\Compute the normal form of $T$.
\begin{itemize}
	\item If the Kronecker normal form of $ T$ is $\begin{bmatrix}
		\lambda & 0 & 0\\
		0 & \lambda & 0 \\
		0 & 0& \mu 
	\end{bmatrix} $ then the output is \emph{T belongs to case \ref{5.} of Theorem \ref{newMAINTHM}, therefore it is not identifiable}. 
	\item Else the output will be the rank of $ T$ computed via \eqref{eqrangopencil} of Theorem \ref{jajaeco} and \emph{$T $ is not on the list of non-identifiable rank-3 tensors}.
\end{itemize}
\item Otherwise $(n_1,n_2,n_3)=(3,3,3)$ and the output is \emph{$T $ is not on the list of non-identifiable rank-3 tensors, hence $T$ is either identifiable or its rank is greater than $3$}.
\end{enumerate}
\end{algorithm}
Here we provide an implementation of Algorithm \ref{part2algo} with the algebra software Macaulay2 \cite{M2}. The input of the function is a concise $3$-factors tensor $T\in \mathbb{C}^{n_1}\otimes \mathbb{C}^{n_2}\otimes \mathbb{C}^{n_3}$, with $2\leq n_1\leq n_2\leq n_3 \leq 3$. In practice $T$ must be given as a list of matrices $\{ A_1,\dots, A_{n_1}\}$, where each $A_i\in M_{n_2\times n_3}(\mathbb{C})$ as displayed in the following image.

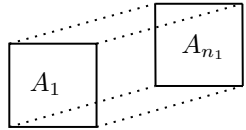
\begin{center}
	\tikzset{every picture/.style={line width=0.75pt}} 
	
	\begin{tikzpicture}[x=0.75pt,y=0.75pt,yscale=-1,xscale=1]
		
		\draw   (322.25,158.65) -- (366.25,158.65) -- (366.25,200.67) -- (322.25,200.67) -- cycle ;
		\draw  [dash pattern={on 0.84pt off 2.51pt}]  (366.25,200.67) -- (440,180) ;
		\draw  [dash pattern={on 0.84pt off 2.51pt}]  (366.25,158.65) -- (439.82,138.5) ;
		\draw  [dash pattern={on 0.84pt off 2.51pt}]  (322.25,158.65) -- (395.81,138.5) ;
		\draw  [dash pattern={on 0.84pt off 2.51pt}]  (322.25,200.67) -- (396,180) ;
		\draw    (395.81,138.5) -- (396,180) ;
		\draw    (396,180) -- (440,180) ;
		\draw    (395.81,138.5) -- (439.82,138.5) ;
		\draw    (439.82,138.5) -- (440,180) ;
		
		\draw (331.05,172.05) node [anchor=north west][inner sep=0.75pt]    {$A_{1}$};
		\draw (407.51,152.79) node [anchor=north west][inner sep=0.75pt]    {$A_{n_{1}}$};

	\end{tikzpicture}
\end{center}
  For the case $(n_1,n_2,n_3)=(2,2,2)$ the algorithm evaluates the Cayley's hyperdeterminant in the entries of the tensor, while for the remaining cases it computes the Kronecker normal form of the matrix pencil associated to the given $T $.
  
  \vspace{1.cm}

\begin{lstlisting}[frame=single,caption={M2 implementation of Algorithm \ref{part2algo} }]
needsPackage "SparseResultants"
needsPackage "Kronecker"

threefactors = (T)->(
	n = (length T,numrows T_0,numcols T_0);
	if n == (2,2,2) then(
    	dis = sparseDiscriminant(genericMultihomogeneousPolynomial((2,2,2),(1,1,1)));
    	V = (ring dis)**ring (T_0);
    	T = apply(T,x->sub(x,V));
    	Tlist = flatten {flatten entries T_0,flatten entries T_1};
    	dis = sub(sub(dis,V),for i to length Tlist-1  list V_(i)=>Tlist_i);
    	return {n,dis};
    )else if n==(3,3,3) then(
    	print "n=(3,3,3), T is not on the list";
    )else( 
    	R = QQ[x,y];
    	T = apply(T,x->sub(x,R));
    	A = kroneckerNormalForm (x*T_0+y*T_1);
    	return {n,A_0};
    );
)

\end{lstlisting}

\subsection{More than three factors}\label{Subsec:+3fattori}

We are now ready to develop the case in which a concise tensor space of a tensor has more than $3$ factors,  i.e. $$ \mathcal{T}_{n_1,\dots,n_{k}}=\mathbb{C}^{n_1}\otimes \cdots \otimes \mathbb{C}^{n_{k}}$$ where $ k>3$ and all $ n_i \in \{2,3\}$. 
We will first treat the case in which $k=4$ and $n_1=n_2=n_3 =n_4=2$ and then we will treat all together the remaining cases.

\subsubsection{Non-identifiable tensors with at least $4$ factors}
Consider for the moment the $4$-factors case, i.e. $$ \mathcal{T}_{n_1,n_2,n_3,n_4}=\mathbb{C}^{n_1}\otimes \mathbb{C}^{n_2}\otimes \mathbb{C}^{n_3}\otimes \mathbb{C}^{n_4},$$ where all $n_i\in \{ 2,3\}$. Following the classification of Theorem \ref{newMAINTHM}, working with $4$ factors there are only two families of non-identifiable tensors, namely items \ref{6.} and \ref{5.}. Case \ref{5.} is referred to non-identifiable rank-$3$ tensors of \cite[Proposition 3.10]{BBS} adapted to the $4$-factors case, while case \ref{6.} contains any rank-3 tensor in $\mathbb{C}^2\otimes \mathbb{C}^2\otimes \mathbb{C}^2\otimes \mathbb{C}^2$. 
Let us first treat the case of $\mathcal{T}_{2^4}=\mathbb{C}^2\otimes \mathbb{C}^2\otimes \mathbb{C}^2\otimes \mathbb{C}^2$.

\bigskip

\noindent
\subsubsection{$\mathcal{T}_{2^4}=\mathbb{C}^2\otimes \mathbb{C}^2\otimes \mathbb{C}^2\otimes \mathbb{C}^2$  }\phantom{a}\\

\noindent
As already recalled, the third secant variety of the Segre variety $X_{1^4}$ is defective (cf. \cite[Theorem 4.5]{AOP}). Moreover, by \cite[Theorem 1.8]{BB3sec}, any element of $ \sigma_3(X_{1^4})\setminus \sigma_2(X_{1^4})$ is a rank-$3$ tensor. 
Therefore any tensor in $\sigma_3(X_{1^4})\setminus \sigma_2(X_{1^4})$ is a non-identifiable rank-$3$ tensor.

Thus, working over $\mathcal{T}_{2^4}$, to detect whether a given tensor $T\in \mathcal{T}_{2^4}$ is a non-identifiable rank-$3$ tensor it is sufficient to verify if $[T]\in \sigma_3(X_{1^{4}})\setminus \sigma_2(X_{1^{4}})$, i.e. if $T$ satisfies the equations of $\sigma_3(X_{1^4})$ (cf. \cite[Theorem 1.4]{Qieqsigma3})
and $T$ does not satisfies the equations of $\sigma_2(X_{1^4})$ for which we refer to \cite{LandMan}.
\bigskip


\noindent
\subsubsection{$\mathcal{T}_{n_1,\dots,n_{k}}\neq \mathbb{C}^2\otimes \mathbb{C}^2\otimes \mathbb{C}^2\otimes \mathbb{C}^2$  , with $k\geq 4$, $n_i=2,3$ for all $i=1,\dots,k$}\phantom{a}\\

\noindent
Let now $k\geq 4$ with $\mathcal{T}_{n_1,\dots,n_{k}}\neq \mathbb{C}^2\otimes \mathbb{C}^2\otimes \mathbb{C}^2\otimes \mathbb{C}^2$. In this case, any non-identifiable rank-$3$ tensor comes from case \ref{5.} of Theorem \ref{newMAINTHM}.
 More precisely, let $$ Y':=\mathbb{P}^1\times \mathbb{P}^1\times \{ u_3 \} \times \cdots \times \{ u_k\}\subset Y_{m_1,m_2,1^{k-2}}=\mathbb{P}^{m_1}\times \mathbb{P}^{m_2}\times \mathbb{P}^1\times \cdots \times \mathbb{P}^1,$$ with $m_1,m_2\in \{ 1,2\}$. Let $q'\in \langle \nu(Y') \rangle \setminus \nu(Y_{m_1,m_2,1^{k-2}}) $ and $ p\in Y_{m_1,m_2,1^{k-2}}\setminus Y'$. We saw that any $[T]\in \langle  q', \nu(p)\rangle $ is a non-identifiable rank-3 tensor. Let $ \{u_i,\tilde{u}_i\}$ be a basis of the $\mathbb{C}^{n_i}$ arising from the $i$-th factor of $Y_{m_1,m_2,1^{k-2}}$ for all $ i\geq 3$. Take distinct $a_1,a_2\in \mathbb{C}^{m_1+1}$ and distinct $b_1,b_2\in \mathbb{C}^{m_2+1}$ and if $m_1=1$ then let $a_3\in \langle a_1,a_2\rangle $ otherwise we let $ a_1,a_2,a_3$ form a basis of the first factor. Let $ b_3 \in \langle b_1,b_2\rangle$ if $m_2=1$, otherwise $ b_1,b_2,b_3$ form a basis of the second factor. With respect to these bases $T$ can be written as
\begin{align}\label{nonidpiufat}
T=& (a_1\otimes b_1+a_2\otimes b_2)\otimes u_3 \otimes \cdots \otimes u_k + a_3 \otimes b_3\otimes \tilde{u}_3 \otimes \cdots \otimes \tilde{u}_k.
\end{align}    
Since the only type of tensors that we have to detect corresponds to \eqref{nonidpiufat}, we may restrict ourselves to consider the following tensor spaces:
\begin{itemize}
\item $ \mathcal{T}_{3,2^{k-1}}=\mathbb{C}^3 \otimes \mathbb{C}^2\otimes \mathbb{C}^2\otimes \cdots \otimes \mathbb{C}^2$;
\item $ \mathcal{T}_{3,3,2^{k-2}}=\mathbb{C}^3 \otimes \mathbb{C}^3 \otimes  \mathbb{C}^2\otimes \cdots \otimes \mathbb{C}^2$;
\item $ \mathcal{T}_{2^{k}}=\mathbb{C}^2 \otimes \mathbb{C}^2\otimes \mathbb{C}^2\otimes \cdots \otimes \mathbb{C}^2$ (with $k\geq 5$).
\end{itemize}

\begin{definition}
Let $\mathcal{T}_{n_1,\dots,n_k}=\mathbb{C}^{n_1}\otimes \cdots \otimes \mathbb{C}^{n_k}$, fix integer $k'\leq k$ and let $I=\cup_{i=1}^{k'} I_i$ be a partition of $\{1,\dots,k\}$.  A \emph{reshape} of $\mathcal{T}$ of type $I_1,\dots,I_{k'} $ is a bijection
\begin{align*}
\theta_{I_1,\dots,I_{k'}} : \mathcal{T}_{n_1,\dots,n_k}\longrightarrow \mathbb{C}^{N_1}\otimes \cdots \otimes \mathbb{C}^{N_{k'}},
\end{align*} 
where $ \mathbb{C}^{N_i}\cong \bigotimes_{j\in I_i}\mathbb{C}^{n_j} $ for all $i=1,\dots,k'$, i.e. $N_i=\prod_{j\in I_i}n_i$ and $\mathbb{C}^{N_i}$ is the vectorization of $\bigotimes_{j\in I_i}\mathbb{C}^{n_j} $.
\end{definition}
In other words a reshape of a tensor space $\mathcal{T}_{n_1,\dots,n_k}$ is a different way of grouping together some of the factors of $\mathcal{T}_{n_1,\dots,n_k}$ (eventually it is also necessary to reorder the factors of $\mathcal{T}_{n_1,\dots,n_k}$).

In the following we will be interested in the reshape grouping together two factors of a tensor space $\mathcal{T}_{n_1,\dots,n_k}$ and we will denote by $\theta_{i,j}$ the corresponding map for some $i,j=1,\dots,k$, i.e.
$$\theta_{i,j} \colon \mathbb{C}^{n_1}\otimes \cdots \otimes \mathbb{C}^{n_{k}}\xrightarrow{\sim} (\mathbb{C}^{n_i}\otimes \mathbb{C}^{n_j})\otimes \mathbb{C}^{n_1}\otimes \cdots \otimes \widehat{\mathbb{C}^{n_i}} \otimes \cdots \otimes \widehat{\mathbb{C}^{n_j}}\otimes \cdots \otimes \mathbb{C}^{n_{k}}.   $$
\begin{example}
Let  $\mathcal{T}_{n_1,\dots,n_k}=\mathbb{C}^{n_1}\otimes \cdots \otimes \mathbb{C}^{n_k}$ and denote by $\theta_{1,2}$ the reshape grouping together the first two factors of $\mathcal{T}_{n_1,\dots,n_k}$
\begin{align*}
\theta_{1,2}: \mathcal{T}_{n_1,\dots,n_k}&\longrightarrow \left(\mathbb{C}^{n_1}\otimes \mathbb{C}^{n_2}\right)\otimes \mathbb{C}^{n_3}\otimes \cdots \otimes \mathbb{C}^{n_k} \\
T=\sum_{\substack{i_1,\dots,i_k \\  i_j=1,\dots,n_j,j=1,\dots,k}}  t_{i_1,\dots,i_k}e_{i_1}\otimes \cdots \otimes e_{i_k}&\mapsto \sum_{\substack{i_1,\dots,i_k \\  i_j=1,\dots,n_j,j=1,\dots,k}}  t_{i_1,\dots,i_l}(e_{i_1}\otimes e_{i_2})\otimes e_{i_3}\otimes \cdots \otimes e_{i_k}.
 \end{align*}

Since $\mathbb{C}^{n_1}\otimes \mathbb{C}^{n_2} \cong \mathbb{C}^{n_1n_2}$, by sending the basis $\{ e_{i_1}\otimes e_{i_2}\}_{i_1=1,\dots,n_1,i_2=1,\dots,n_2}$ of $\mathbb{C}^{n_1}\otimes \mathbb{C}^{n_2}$ to the basis $\{ e_{i_1,i_2}\}$ of $\mathbb{C}^{n_1n_2}$, we write $$
\theta_{1,2}(T)= \sum_{ i_1,\dots,i_k }   t_{i_1,i_2,i_3,\dots,i_k} e_{i_1,i_2}\otimes e_{i_3}\otimes \cdots \otimes e_{i_k}\in \mathbb{C}^{n_1n_2}\otimes \mathbb{C}^{n_3}\otimes \cdots \otimes \mathbb{C}^{n_k}.$$

\end{example}

The following lemma tells us how to completely characterize non-identifiable rank-$3$ tensors lying on either $\mathcal{T}_{3,2^{k-1}}$ or $\mathcal{T}_{3,3,2^{k-2}}$ or $\mathcal{T}_{2^k}$.

\begin{lemma}\label{lemmaalgo}
Let $T\in \mathcal{T}_{n_1,n_2,2^{k-2}}=\mathbb{C}^{n_1} \otimes \mathbb{C}^{n_2}\otimes \mathbb{C}^2\otimes \cdots \otimes \mathbb{C}^2$ be a concise tensor in $ \mathcal{T}_{n_1,n_2,2^{k-2}} $, where $n_1,n_2\in \{2,3\}$, $k\geq 4$ and $ \mathcal{T}_{n_1,n_2,2^{k-2}}\neq \mathcal{T}_{2^4}$. Then $T$ is as in case \ref{5.} of Theorem \ref{newMAINTHM} if and only if the following conditions hold:
\begin{enumerate}
\item the reshaped tensor $\theta_{1,2}(T)\in \mathbb{C}^{n_1n_2}\otimes (\mathbb{C}^{2})^{\otimes (k-2)}$ is an identifiable rank-$2$ tensor with respect to $ \mathbb{C}^{n_1n_2}\otimes (\mathbb{C}^{2})^{\otimes(k-2)}$
$$\theta_{1,2}(T)=T_1+T_2=x\otimes u_3\otimes \cdots \otimes u_k + y \otimes v_3\otimes \cdots \otimes v_k \in \mathbb{C}^{n_1n_2}\otimes (\mathbb{C}^{2})^{\otimes (k-2)}$$
for some independent $x,y\in \mathbb{C}^{n_1n_2}$ and some $u_i,v_i\in \mathbb{C}^2$ with $\{u_i,v_i\}$ linearly independent for all $i=3,\dots,k$;
\item looking at $x,y \in \mathbb{C}^{n_1n_2}$ as elements of $\mathbb{C}^{n_1}\otimes \mathbb{C}^{n_2}$ then $\{r(x),r(y)\}=\{1,2\}$.
\end{enumerate}
\end{lemma}

\begin{proof}
Let $T\in \mathcal{T}_{n_1,n_2,2^{k-2}}$ be as in case \ref{5.} of Theorem \ref{newMAINTHM}, so $T$ can be written as 
$$T=a_1\otimes b_1\otimes u_3\otimes \cdots \otimes u_{k} + a_2\otimes b_2\otimes u_3\otimes \cdots \otimes u_k + a_3 \otimes b_3 \otimes v_3\otimes \cdots \otimes v_k ,  $$
where $u_i\neq v_i $ for all $i=3,\dots, k$, $a_1,a_2,a_3$ are linearly independent if $n_1=3$ and $b_1,b_2,b_3$ are linearly independent if $n_2=3$. Let $\theta_{1,2}$ be the reshape grouping together the first two factors of $\mathcal{T}_{n_1,\dots,n_k}$. Let $x:=a_1\otimes b_1, y:= a_2\otimes b_2$ and $z:=a_3\otimes b_3$ and remark that $r(x+y)=2$ and $r(z)=1$. Therefore
\begin{align*}
 \theta_{1,2}(T)=&x\otimes u_3\otimes \cdots \otimes u_k + y\otimes u_3\otimes \cdots \otimes u_k + z\otimes v_3\otimes \cdots \otimes v_k\\ 
=& (x+y)\otimes u_3\otimes \cdots \otimes u_k + z\otimes v_3\otimes \cdots \otimes v_k\\
=& T_1 + T_2 \in \mathbb{C}^{n_1n_2}\otimes \mathbb{C}^2\otimes \cdots \otimes \mathbb{C}^2.
\end{align*} 
Note that the rank of $(T_1+T_2)\in \mathcal{T}_{n_1n_2,2^{k-2}}$ is at most 2 and in fact $r(T_1+T_2)=2$ since $u_i,v_i$ are linearly independent for all $i=3,\dots,k$. Moreover, we recall that the only non-identifiable rank-2 tensors are matrices (cf. \cite[Proposition 2.3]{BBS}). Therefore, since the concise tensor space of $T_1+T_2$ is made by at least $3$ factors, then $T_1+T_2$ is an identifiable rank-2 tensor.
\smallskip

Viceversa let $T \in \mathcal{T}_{n_1,n_2,2^{k-2}}$ such that $\theta_{1,2}(T)\in \mathbb{C}^{n_1n_2}\otimes (\mathbb{C}^2)^{k-2}$ is an identifiable rank-$2$ tensor 
$$\theta_{1,2}(T)=T_1+T_2=a\otimes u_3\otimes \cdots \otimes u_k+b\otimes v_3\otimes \cdots \otimes v_k, $$
for some unique $a,b\in \mathbb{C}^{n_1n_2}$ with $\langle a, b\rangle \cong \mathbb{C}^2$ and unique $u_i,v_i\in \mathbb{C}^2$ with $\langle u_i, v_i\rangle \cong \mathbb{C}^2$ for all $i=3,\dots,k$. By assumption $\theta^{-1}_{1,2}(a),\theta^{-1}_{1,2}(b)\in \mathbb{C}^{n_1}\otimes \mathbb{C}^{n_2}$ are such that $\{r(\theta^{-1}_{1,2}(a)),r(\theta^{-1}_{1,2}(b))\} = \{ 1,2\}$ and by relabeling if necessary we may assume $r(\theta^{-1}_{1,2}(a))=2$ and $r(\theta^{-1}_{1,2}(b))=1$.

Let us see $\theta_{1,2}(T)$ as an element of $\mathcal{T}_{n_1,n_2,2^{k-2}}=\mathbb{C}^{n_1}\otimes \mathbb{C}^{n_2}\otimes \mathbb{C}^2\otimes \cdots \otimes \mathbb{C}^2$.
Since $T_2$ is a rank-1 tensor, there exist $v_1\in \mathbb{C}^{n_1}$, $v_2\in \mathbb{C}^{n_2}$ such that $\theta^{-1}_{1,2}(b)=v_1\otimes v_2$, i.e. 
$$\theta^{-1}_{1,2}(T_2)=v_1\otimes v_2\otimes v_3\otimes \cdots \ \otimes v_k .$$
 
Moreover, since $r(\theta^{-1}_{1,2}(a))=2$ then there exist linearly independent $a_1,a_2\in \mathbb{C}^{n_1}$ and linearly independent $b_1,b_2\in \mathbb{C}^{n_2}$ such that $\theta^{-1}_{1,2}(a)=a_1\otimes b_1+a_2\otimes b_2$, i.e.
$$ \theta^{-1}_{1,2}(T_1)= a_1\otimes b_1\otimes u_3\otimes \cdots \otimes u_k + a_2\otimes b_2\otimes u_3\otimes \cdots \otimes u_k.$$
We remark that the concise space of $T $ is $\mathcal{T}_{n_1,n_2,2^{k-2}}$, therefore if $n_1=3$ (or $n_2=3$) then $a_1,a_2, v_1$ are linearly independent ($b_1,b_2,v_2$ are linearly independent). Thus $T$ is as in case \ref{5.}.
\end{proof}

\begin{remark}\label{rempermutaz}
In Lemma \ref{lemmaalgo} we assumed that dealing with a tensor as in \eqref{nonidpiufat} the non-identifiable part of the tensor was in the first two factors because it is always possible to permute the factors of the tensor space in this way. This assumption cannot be made in the algorithm and we have to be careful if either $(n_1,n_2)=(3,2)$ or $(n_1,n_2)=(2,2)$.
Dealing with $(n_1,n_2)=(3,2)$, we have to check if there exists $i=2,\dots,k$ such that $\theta_{1,i}(T)$ satisfies the conditions of Lemma \ref{lemmaalgo}.\\
Similarly, for the case of $(n_1,n_2)=(2,2)$ we have to check all reshaping of $T$ if necessary, i.e. we have to check if there exist $i,j\in \{1,\dots,k\}$ with $i\neq j$ such that $\theta_{i,j}(T)$ satisfies the conditions of Lemma \ref{lemmaalgo}.
 \end{remark}

Recall that a concise tensor $T\in \mathbb{C}^{n_1n_2}\otimes (\mathbb{C}^2)^{\otimes (k-2)}$ is an element of $\sigma_2(X_{(n_1n_2-1),1^{k-2}})\setminus \tau(X_{(n_2n_2-1),1^{k-2}})$ if and only if there is a specific change of basis on each factors $\tilde{g}=(g,g_3,\dots,g_k)\in GL_{n_1n_2}\times GL_2\times \cdots \times GL_2$  such that 
\begin{equation}\label{r2reshape}
\tilde{g}\cdot T=x\otimes u_3\otimes \cdots u_k + y \otimes v_3\otimes \cdots \otimes v_k. 
\end{equation}
By Lemma \ref{lemmaalgo}, given an identifiable rank-2 tensor $T\in \mathcal{T}_{n_1n_2,2^{k-2}}$, in order to verify if $T$ is as in case \ref{5.}, we do not need to find an explicit decomposition of $T$ as in \eqref{r2reshape} but it is enough made the following steps:
\begin{itemize}
 \item distinguish $x,y\in \mathbb{C}^{n_1n_2}$ and look at them as elements of $\mathbb{C}^{n_1}\otimes \mathbb{C}^{n_2}$; 
\item prove that either $r(x)=2$ and $r(y)=1$ or that $r(x)=1$ and $r(y)=2$. 
\end{itemize}
Let us explain in detail how to do so.
\subsubsection{Reshape procedure for an identifiable rank-2 tensor of $\mathcal{T}_{n_1n_2,2^{k-2}}$ (how to find $x,y\in \mathbb{C}^{n_1}\otimes \mathbb{C}^{n_2}$)}
Let $T$ be an identifiable rank-$2$ tensor in $\mathcal{T}_{n_1n_2,2^{k-2}}=\mathbb{C}^{n_1n_2}\otimes (\mathbb{C}^2)^{\otimes (k-2)}$. Remark that the rank of the first flattening $\varphi_1\colon (\mathbb{C}^2)^{\otimes (k-2)}\rightarrow  (\mathbb{C}^{n_1n_2})^*$ of $T$ is $2$ and to complete the concision process, we can take as basis of the first new factor two independent elements $\widehat{x},\widehat{y} $ of  $\hbox{Im}(\varphi_1)$.
Therefore $T$ can be written as
$$T=\widehat{x}\otimes u_3\otimes \cdots \otimes u_k + \widehat{y} \otimes v_3\otimes \cdots \otimes v_k\in \mathbb{C}^2\otimes (\mathbb{C}^2)^{\otimes k}. $$ 
If we reshape our tensor space by grouping together all factors from the $4$-th one onwards, then $T$ can be seen as
\begin{align*}
\widehat{x} \otimes u_3\otimes \overbrace{(u_4\otimes \cdots \otimes u_k)}^{\widehat{u}}+\widehat{y}\otimes v_3\otimes\overbrace{(v_4\otimes \cdots \otimes v_k)}^{\widehat{v}}=&\\
\widehat{x} \otimes u_3\otimes \widehat{u} + \widehat{y}\otimes v_3\otimes \widehat{v}&\in \mathbb{C}^2\otimes \mathbb{C}^2\otimes ((\mathbb{C}^2)^{\otimes (k-3)})  .
\end{align*}

We want to look at this 3-factors tensor as a pencil of matrices with respect to the second factor of $\mathbb{C}^2\otimes \mathbb{C}^2\otimes (\mathbb{C}^{2})^{\otimes (k-3)}$. Let $u_3=(u_{3,1},u_{3,2})$, $v_3=(v_{3,1},v_{3,2})$ and denote by  $$C_1:=\begin{bmatrix} u_{3,1} \widehat{x}\otimes \widehat{u} \\ v_{3,1} \widehat{y}\otimes \widehat{v}\end{bmatrix} , \; C_2:= \begin{bmatrix} u_{3,2} \widehat{x}\otimes \widehat{u} \\ v_{3,2} \widehat{y}\otimes \widehat{v}\end{bmatrix}\in \mathbb{C}^2\otimes (\mathbb{C}^{2})^{\otimes(k-3)}.$$
We can write $T$ as
$$ C_1\lambda +C_2\mu.$$
Call $X_3$ the matrix whose columns are given by $\widehat{x}$ and $\widehat{y}$ and denote by $X_4$ the matrix whose rows are given by $\widehat{u} $ and $\widehat{v}$. Therefore
\begin{align*}
C_1 = \begin{bmatrix} \widehat{x} & \widehat{y} \end{bmatrix}\begin{bmatrix} u_{3,1} & 0 \\ 0 & v_{3,1} \end{bmatrix} \begin{bmatrix} \widehat{u} \\ \widehat{v} \end{bmatrix}= X_3 \begin{bmatrix} u_{3,1}&0\\ 0 & v_{3,1}\end{bmatrix}X_4, \\
C_2 = \begin{bmatrix} \widehat{x} & \widehat{y} \end{bmatrix}\begin{bmatrix} u_{3,2} & 0 \\ 0 & v_{3,2} \end{bmatrix} \begin{bmatrix} \widehat{u} \\ \widehat{v} \end{bmatrix}= X_3  \begin{bmatrix} u_{3,2}&0\\ 0 & v_{3,2}\end{bmatrix}X_4. 
\end{align*}
Remark that $C_2$ is right invertible and denote by $C_2^{-1}$ its right inverse. Moreover $r(X_3)=r(X_4)=2$, therefore $X_3$ is invertible and there exists a right inverse of $X_4$ that we denote by $X_4^{-1}$. Thus
\begin{align*}
C_1C_2^{-1}&=\left( X_3 \begin{bmatrix} u_{3,1}&0\\ 0 & v_{3,1}\end{bmatrix}X_4\right)\left(X_3  \begin{bmatrix} u_{3,2}&0\\ 0 & v_{3,2}\end{bmatrix}X_4\right)^{-1}\\
&= X_3 \begin{bmatrix} \frac{u_{3,1}}{u_{3,2}}&0\\ 0 & \frac{v_{3,1}}{v_{3,2}}\end{bmatrix}X_3^{-1}.
\end{align*}
We have now an eigenvalue problem that we can easily solve to find $\widehat{x},\widehat{y}\in \mathbb{C}^2$.

\begin{remark}
When computing the concision process of $T$ with respect to the first factor of $\mathcal{T}_{n_1n_2,2^{k-2}}$, we concretely find a basis of $\hbox{Im}(\varphi_1)$. Therefore, after we found $\widehat{x},\widehat{y}\in \mathbb{C}^2$ with the above procedure, we can easily get back to $x,y\in \mathbb{C}^{n_1n_2}\cong \mathbb{C}^{n_1}\otimes \mathbb{C}^{n_2}$ and compute the rank of both $x,y$ seen as elements of $\mathbb{C}^{n_1}\otimes \mathbb{C}^{n_2}$.
\end{remark}
%

%

Let us sum up how to find a non-identifiable rank-3 tensor of at least $4$ factors in the following pseudo-algorithm.

\begin{algorithm}[H]
	\caption{(Non-identifiability with at least $4$ factors)}\label{part3algo} 
	\noindent 
\begin{flushleft}
\textbf{Input:} Concise tensor $ T=(t_{i_1,\dots, i_{k}})\in \mathcal{T}_{n_1,n_2,2^{k-2}}$, 
for some $k>3$, $2\leq n_1,n_2\leq 3$.\\
\textbf{Output:} A statement on whether $T$ either belongs to one of the six cases of non-identifiable rank-3 tensors or not.
\end{flushleft}
\begin{enumerate}\setcounter{enumi}{-1}
\item For all $i,j=1,\dots,k$ with $i\neq j$ denote by $\theta_{i,j}$ the reshape grouping the $i$-th and $j$-th factor of $\mathcal{T}_{n_1,\dots,n_{k}}$.
\item  Case $(n_1,n_2)=(2,2)$. 
\begin{itemize}
\item Case $k=4$. Test if $T\in \sigma_3(X_{1^4})\setminus \sigma_2(X_{1^4})$ (cf. \cite[Theorem 1.4]{Qieqsigma3} for the equations of the third secant variety and  \cite{LandMan} for the equations of the second secant variety). If the answer to the test is positive, the output is: \emph{$T$ is a non-identifiable rank-$3$ tensor}, otherwise the output is: \emph{$T $ is not on the list of non-identifiable rank-$3$ tensors}.
\item Case $k\geq 5$. For all $i=1,\dots,k-1$ and for all $j=i+1\dots,k$ follow this procedure:
\begin{itemize}

\item[$\bullet$] Test if $\theta_{i,j}(T)$ satisfies the equations of $\sigma_2(X_{3,1^{k-2}})$ and does not satisfy the equations of $\tau(X_{3,^{k-2}})$ (cf. \cite{LandMan}, \cite[Theorem 1.3]{OedingTan} for equations of both varieties). If $\theta_{i,j}(T)\in \sigma_2(X_{3,1^{k-2}})\setminus \tau(X_{3,1^{k-2}})$ then $\theta_{i,j}(T)$ is an identifiable rank-$2$ tensor. 

Make the concision process on the first factor of $\mathcal{T}_{3,1^{k-2}}$ and call $T'$ the resulting tensor. 

Consider $T'$ as a matrix pencil of $\mathbb{C}^2\otimes \mathbb{C}^2\otimes( (\mathbb{C}^2)^{\otimes(k-2)})$ with respect to the second factor
$$T'=\lambda C_1 + \mu C_2. $$
Find the eigenvectors $x,y\in \mathbb{C}^2$ of $C_1C_2^{-1}$ and then rewrite $x,y$ as elements of $\mathbb{C}^{4}\cong \mathbb{C}^2\otimes \mathbb{C}^2$ via $\theta^{-1}_{i,j}$. If $\{r(x),r(y) \}=\{ 1,2\}$ then the output is: \emph{$T$ is a non-identifiable rank-$3$ tensor corresponding to case \ref{5.} of Theorem \ref{newMAINTHM}}.
\item[$\bullet$] Else, if one of the previous conditions is not satisfied, then stop and restart with another $j$ (and another $i$ when necessary).

\end{itemize}
If the algorithm stops at some point when $i=k-1,j=k$ then break and the output is: \emph{$T$ is not on the list of non-identifiable rank-$3$ tensors}.

\end{itemize}

\item Case $(n_1,n_2)=(3,2)$.\\
For all $i=2,\dots,k-1$ follow this procedure:
\begin{itemize}
\item Test if $\theta_{1,i}(T)$ satisfies the equations of $\sigma_2(X_{5,1^{k-2}})$ and does not satisfy the equations of $\tau(X_{5,^{k-2}})$ (cf. \cite{LandMan}, \cite[Theorem 1.3]{OedingTan} for equations of both varieties). If $\theta_{1,i}(T)\in \sigma_2(X_{5,1^{k-2}})\setminus \tau(X_{5,1^{k-2}})$ then $\theta_{1,i}(T)$ is an identifiable rank-$2$ tensor. Reduce the first factor of $\mathcal{T}_{6,2^{k-2}}$ via concision, working now on $\mathcal{T}_{2^{k-1}}$ with $T'$. Consider $T'$ as a matrix pencil with respect to the second factor of $\mathbb{C}^2\otimes \mathbb{C}^2\otimes (\mathbb{C}^2)^{\otimes (k-3)}$, i.e. 
$$ T'=\lambda C_1 + \mu C_2.$$ 
Find the eigenvectors $x,y$ of $C_1C_2^{-1}$ and then rewrite $x,y$ as elements of $\mathbb{C}^6=\mathbb{C}^3\otimes 	\mathbb{C}^2$ via $\theta^{-1}_{1,i}$. If $\{ r(x),r(y)\}=\{ 2,1 \}$ the output is: \emph{$T$ is a non-identifiable rank-$3$ tensor}.
\item Else, if one of the previous conditions is not satisfied then stop and restart with another $i$.

\end{itemize}
If the algorithm stops at some point when $i=k$ then break and the output is: \emph{$T$ is not on the list of non-identifiable rank-$3$ tensors}.



	\item Case $(n_1,n_2)=(3,3)$.\\
	\begin{itemize}
	\item Test if $\theta_{1,2}(T)$ satisfies the equations of $\sigma_2(X_{8,1^{k-2}})$ and does not satisfy the equations of $\tau(X_{8,^{k-2}})$ (cf. \cite{LandMan}, \cite[Theorem 1.3]{OedingTan} for equations of both varieties). If $\theta_{1,2}(T)\in \sigma_2(X_{8,1^{k-2}})\setminus \tau(X_{8,1^{k-2}})$ then $\theta_{1,2}(T)$ is an identifiable rank-2 tensor. Reduce the first factor of $\mathcal{T}_{9,2^{k-2}}$ via the concision process, working now with $T'$ on $(\mathbb{C}^2)^{\otimes (k-1)}$. Consider $T'$ as a matrix pencil with respect to the second factor of $\mathbb{C}^2\otimes \mathbb{C}^2\otimes (\mathbb{C}^2)^{\otimes (k-3)}$, i.e. 
	$$ T'=\lambda C_1+\mu C_2.$$
	Find the eigenvectors $x,y$ of $C_1C_2^{-1}$ and then rewrite $x,y$ as elements of $\mathbb{C}^9\cong \mathbb{C}^3\otimes \mathbb{C}^3$ via $\theta^{-1}_{1,2}$. If $\{ r(x),r(y) \}=\{1,2 \}$ the output is: \emph{$T$ is a non-idenfitiable rank-$3$ tensor as in case \ref{5.}.}
	\item If one of these conditions is not satisfied then stop and the output is: \emph{$T$ is not on the list of non-identifiable rank-$3$ tensors}.
	\end{itemize}
\end{enumerate}
\end{algorithm}

	A code implementation in \texttt{Macaulay2} of the above algorithm is available at the repository website MathRepo of MPI-MiS
	via the link 
	\begin{center}
\url{https://mathrepo.mis.mpg.de/identifiabilityRank3tensors}.
\end{center}

\begin{example}
Let $ \mathcal{T}_{3,2,2,2}=\mathbb{C}^3\otimes \mathbb{C}^2\otimes \mathbb{C}^2\otimes \mathbb{C}^2$ and for all $j,k,\ell= 1,2$ and for all $i=1,2,3$ denote $e_{i,j,k,\ell}=e_i\otimes e_j\otimes e_k \otimes e_\ell$. To lighten the notation we also set $e_{i}e_{j}=e_i\otimes e_j$. Consider the tensor
\begin{align*}
T= &12 e_{1,1,1,1}+ 8e_{1,1,1,2}+ 6 e_{1,1,2,1}+4 e_{1,1,2,2}+30e_{1,2,1,1}+20e_{1,2,1,2}+15e_{1,2,2,1}+
\\&10e_{1,2,2,2}+8e_{2,1,1,1}+8e_{2,1,1,2}+5e_{2,1,2,1}+6e_{2,1,2,2}+35e_{2,2,1,1}+38e_{2,2,1,2}+\\
&23e_{2,2,2,1} +   30e_{2,2,2,2} +   16e_{3,1,1,1}   + 16e_{3,1,1,2}    +10e_{3,1,2,1}    +12e_{3,1,2,2}    +\\   &52e_{3,2,1,1}    +    64e_{3,2,1,2}   +   37e_{3,2,2,1}   +   54e_{3,2,2,2}.
\end{align*}
Let $\theta_{1,2}\colon \mathcal{T}_{3,2,2,2}\rightarrow \mathbb{C}^6\otimes \mathbb{C}^2\otimes \mathbb{C}^2$ be the reshape grouping together the first two factors of $\mathcal{T}_{3,2,2,2}$. Let $$ \theta_{1,2}(e_1e_1)=e_{1,1},\; \theta_{1,2}(e_1e_2)=e_{1,2},\; \theta_{1,2}(e_2e_1)= e_{2,1} ,\; \theta_{1,2}(e_2e_2) = e_{2,2} ,\; \theta_{1,2}(e_3e_1) = e_{3,1} ,\;  \theta_{1,2}(e_3e_2)=e_{3,2}$$ 
be a basis of $\mathbb{C}^6$ such that $\theta_{1,2}(T)$ can be written as
\begin{align*}
 \theta_{1,2}(T)= &12 e_{1,1}\otimes e_1e_1+ 8 e_{1,1} \otimes e_1e_2+ 6e_{1,1}\otimes e_2e_1+4 e_{1,1}\otimes e_2e_2+30e_{1,2}\otimes e_1 e_1+20e_{1,2}\otimes e_1 e_2+\\ &15 e_{1,2} \otimes e_2 e_1+
10 e_{1,2} \otimes e_2 e_2+8 e_{2,1}\otimes e_1 e_1+8 e_{2,1} \otimes e_1 e_2+5 e_{2,1} \otimes e_2 e_1+6 e_{2,1} \otimes e_2e_2+\\ &35 e_{2,2}\otimes e_1 e_1+38 e_{2,2} \otimes e_1 e_2+23 e_{2,2} \otimes e_2 e_1 +   30 e_{2,2} \otimes e_2 e_2 +   16 e_{3,1} \otimes e_1 e_1   + 16 e_{3,1} \otimes e_1e_2    +\\&10 e_{3,1} \otimes e_2 e_1   +12 e_{3,1} \otimes e_2 e_2    +52 e_{3,2} \otimes e_1 e_1    +    64 e_{3,2} \otimes e_1 e_2   +   37 e_{3,2} \otimes e_2e_1   +   54 e_{3,2} \otimes e_{2}e_2.
\end{align*}
One can verify that $\theta_{1,2}(T)\in \sigma_2(X_{5,1^{3}})\setminus \tau(X_{5,1^{3}})$, therefore we can continue our procedure by considering the matrix associated to the first flattening $\varphi_1\colon (\mathbb{C}^2\otimes \mathbb{C}^2 )^*\rightarrow\mathbb{C}^6$ of $T$:
$$A=
\begin{bmatrix}
12 &   8& 6 & 4 \\
30 &  20& 15 & 10   \\
8 &   8 & 5 & 6    \\
35 &  38 & 23 & 30     \\
16 &  16 & 10 & 12    \\
52& 64 & 37 & 54
\end{bmatrix}.
$$
The rank of $A$ is $2$ and we take the first two columns $\widehat{x},\widehat{y} $ of $A$ as linearly independent vectors of $Im(\varphi_1)$ and rewrite all the others as a linear combinations of $\widehat{x},\widehat{y}$. Denote by $T'$ the resulting tensor
$$ 
T'= \widehat{x} \otimes e_1e_1 + \widehat{y} \otimes e_1e_2 + \left(\frac{1}{4}\widehat{x} +\frac{3}{8} \widehat{y} \right) \otimes e_2 e_1 + \left( -\frac{1}{2} \widehat{x} + \frac{5}{4}\widehat{y}  \right) \otimes e_2 e_2.
$$
 Let us consider now $T'\in \mathbb{C}^2\otimes \mathbb{C}^2 \otimes \mathbb{C}^2$ as a matrix pencil with respect to the second factor
$$ 
T'=\lambda \begin{bmatrix} 1 & 0\\ 0 & 1\end{bmatrix} + \mu \begin{bmatrix}1/4 & -1/2 \\ 3/8 & 5/4 \end{bmatrix}=\lambda C_1 + \mu C_2. 
$$
It is easy to see that the eigenvectors of 
 $$
C_1C_2^{-1}=\begin{bmatrix} 10/4 & 1\\ -3/4 &1/2 \end{bmatrix} 
$$ 
are $ x=(-2,1) $ and $y=(-2/3, 1)$, i.e.
\begin{align*}
x=&-2\widehat{x}+\widehat{y} = -(16e_{1,1}+40e_{1,2}+8e_{2,1}+32e_{2,2}+16e_{3,1}+40e_{3,2})=-\begin{bmatrix} 16 & 40 \\
8 & 32 \\
16 & 40 \end{bmatrix}
\end{align*}
and
\begin{align*}
y=&-2/3\widehat{x} + \widehat{y} = 
 8/3e_{2,1}+44/3e_{2,2}+16/3e_{3,1}+88/3e_{3,2}=
\begin{bmatrix}
0&0  \\
8/3 & 44 /3 \\
16/3 & 88 /3
\end{bmatrix}.
\end{align*}
It is easy to see that $r(x)=2$ and $r(y)=1$, therefore $T$ is a non-identifiable rank-3 tensor as in case \ref{5.}. Indeed by multiplying $T$ with 
$$
g= \left( \begin{bmatrix} 1/2 & -1 & 1/2 \\0 & 2 & -1 \\ -1/2 & 0 & 1/2 \end{bmatrix}, \begin{bmatrix} 1 & 0 \\ -1/3 & 1/3\end{bmatrix}, \begin{bmatrix} 1&-1\\-1&2\end{bmatrix}, \begin{bmatrix} 1/2& -1/4 \\-1/2 & 3/4\end{bmatrix} \right) 
$$
we get $$ T=e_1\otimes e_1 \otimes e_1 \otimes e_1 +e_2 \otimes e_2 \otimes e_1 \otimes e_1 + e_3 \otimes (2e_1+3e_2) \otimes e_2 \otimes e_2 .$$
\end{example}
\begin{remark}
Since we already considered all concise spaces of tensors related to all non-identifiable rank-3 tensors of Theorem \ref{newMAINTHM}, any other concise tensor space will not be considered. Therefore, for any other concise space, the output of the algorithm will be \emph{$T$ is not on the list of non-identifiable rank-$3$ tensors}.
\end{remark}

We conclude by collecting all together the steps made until now. 

\begin{algorithm}
	\caption{(Non-identifiable rank-3 tensors)}\label{algoalgo}
\begin{flushleft}
\textbf{Input:} Tensor $ T=(t_{i_1,\dots, i_{k}})\in  \mathbb{C}^{n_1} \otimes \cdots \otimes \mathbb{C}^{n_k}$, for some $k\geq 3$.\\
\textbf{Output:} A statement on whether $T$ belongs to one of the six cases of non-identifiable rank-3 tensors or not.
\end{flushleft}
\begin{enumerate}
\item Compute the concise tensor space $\mathcal{T}_{n'_1,\dots,n'_{k'}}$ of $T$.
\item
\hskip-.4cm\begin{tabular}{l p{16cm} }
\ldelim\{{5}{.1cm}[]
  & If $k'=3$ run Algorithm \ref{part2algo}. \\
 &\\
  & Else if $\mathcal{T}_{n'_1,\dots,n'_{k'}}\in \{ \mathcal{T}_{3,2^{k'-1}}, \mathcal{T}_{3,3,2^{k'-2}},\mathcal{T}_{2^{k'}} \}$, where $ k'\geq 4$, run Algorithm \ref{part3algo}.\\
& \\
& Else the output will be \emph{$T$ is not on the list of Theorem \ref{newMAINTHM}}.
\end{tabular}
\end{enumerate}

\end{algorithm}

\subsection*{Acknowledgement}
This article is part of my Ph.D thesis. I thank my supervisor Alessandra Bernardi for her guidance as well as the many helpful discussions. I would also like to thank Edoardo Ballico, Luca Chiantini and Alessandro Gimigliano for their constructive inputs and Reynaldo Staffolani for the help with the coding part.

\section{Appendix (with E. Ballico and A. Bernardi)}
The purpose of this appendix is to fix an imprecision in the statement of Proposition 3.10 of \cite{BBS}.
Originally stated for $k\geq 3$ factors, Proposition 3.10 of \cite{BBS} describes a family of non-identifiable rank-3 tensors for an arbitrary number of factors and it represents the last item of the classification \cite[Theorem 7.1]{BBS} of identifiable rank-3 tensors. Since Theorem 7.1 of \cite{BBS} is the theoretic basis on which the present paper is based on, we decided to report here the rectification of \cite[Proposition 3.10]{BBS}.

The main issue is that the case $Y_{2,1,1}$ is already completely described by \cite[Examples 3.6 and 3.7]{BBS}, so it does not fall into \cite[case 6, Theorem 7.1]{BBS} but it is already included in cases $4$ and $5$ of the same theorem. In order to fix Theorem 7.1 as stated in \cite{BBS} is therefore sufficient to remove the possibility of $k=3$ and $(n_1,n_2,n_3)=(2,1,1)$from it, case 6, which will remain the same for $k\geq 4$ only.
The statement of \cite[Theorem 7.1]{BBS} becomes as follows.

\begin{theorem}(\protect{\cite[Theorem 7.1]{BBS} revised})\label{newMAINTHM}
	Let $q\in \langle  X_{n_1,\dots,n_k}\rangle $ be a concise rank-3 tensor. Denote by $\mathcal{S}(Y_{n_1,\dots,n_k},q)$ the set of all subsets of $Y_{n_1,\dots,n_k}$ computing the rank of $q$. The rank-3 tensor $q$ is identifiable except in the following cases:
		\begin{enumerate}[label=\alph*)]
			\item\label{1.} $q$ is a rank-3 matrix, in this case $\dim(\mathcal{S}(Y_{2,2},q))=6$;
			\item\label{2.} $q$ belongs to a tangent space of the Segre embedding of $Y_{1,1,1}=\mathbb{P}^1\times \mathbb{P}^1\times\mathbb{P}^1$ and in this case $\dim(\mathcal{S}(Y_{1,1,1,1},q))\geq 2 $;
			\item\label{6.} $q$ is an order-4 tensor of $\sigma_3^0(Y_{1,1,1,1})$ with $Y_{1,1,1,1}=\mathbb{P}^1\times\mathbb{P}^1\times\mathbb{P}^1\times\mathbb{P}^1$, in this case $\dim(\mathcal{S}(Y,q))\geq 1 $. 
			\item\label{3.} $q$ is as in \cite[Example 3.6]{BBS} where $Y_{2,1,1}=\mathbb{P}^2\times\mathbb{P}^1\times\mathbb{P}^1$, in this case $\dim(\mathcal{S}(Y_{2,1,1},q))=3 $;
			\item\label{4.} $q$ is as in \cite[Example 3.7]{BBS} where $Y_{2,1,1}=\mathbb{P}^2\times\mathbb{P}^1\times\mathbb{P}^1$, in this case $\dim \mathcal{S}(Y_{2,1,1},q) =4$;
			\item\label{5.} $q $ is as in Proposition \ref{prop:new statement} where $ Y_{n_1,\dots,n_k}=\mathbb{P}^{n_1}\times \cdots \times \mathbb{P}^{n_k}$ is such that either $k\ge 4$,  $n_i\in \{1,2\}$ for $i=1,2$, $n_i=1$ for $i>2$, or $k=3$ and $(n_1,n_2,n_3)=(2,2,1)$. In this case $\dim(\mathcal{S}(Y_{n_1,\dots,n_k},q))\geq 2 $ and if $ n_1+n_2+k\geq 6$ then $\dim(\mathcal{S}(Y_{n_1,\dots,n_k},q))=2 $. 
			
		\end{enumerate}
	\end{theorem}

This result will be clear after having revised \cite[Proposition 3.10]{BBS}.


Before proceeding, we need to recall the following.

\begin{definition}
	Given $q\in \mathbb{P}^N=\langle X_{n_1,\dots,n_k} \rangle $ the \emph{space of solution} of $q$ with respect to $X_{n_1,\dots,n_k}$ is
	$$
	\mathcal{S}(Y_{n_1,\dots,n_k},q)=\{ A\subset Y_{n_1,\dots,n_k} \colon \# A=r(q) \hbox{ and } q \in \langle \nu(A) \rangle  \}.
	$$
	\end{definition}

\begin{notation}\label{pi}
	We denote the projection on the $i $-th factor as $$\pi_i: Y_{n_1,\dots, n_k} \longrightarrow \mathbb{P}^{n_i} .$$
\end{notation}
	Let us start by considering the case $k=3$ and $(n_1,n_2,n_3)=(2,2,1)$.

	\begin{lemma}(\protect{Case $k=3$})\label{k=3unico} 
		Let $Y_{2,2,1} =\mathbb{P}^2\times \mathbb{P}^2\times \mathbb{P}^1$. Fix two lines $L, R\subset \mathbb{P}^2$, a point $o\in \mathbb{P}^1$ and set $Y':= L\times R\times \{o\}\subset Y_{2,2,1}$. Take $p\in Y_{2,2,1}$ with $\pi _i(p)\notin \pi_i(Y')$ for $i=1,2,3$, i.e. assume that $Y_{2,2,1}$ is the minimal multiprojective space containing $p\cup Y'$. Fix $q'\in \langle \nu(Y')\rangle$ of rank $2$ and $q\in \langle \{ \nu(p),q'\}\rangle$ of rank $3$. Then $\mathcal{S}(Y_{2,2,1},q) =\{\{p\}\cup A\}_{A\in \mathcal{S}(Y',q')}$. 

\end{lemma}

\begin{proof}
	Fix a solution $E\in \mathcal{S}(Y_{2,2,1},q)$. Concision gives $\langle \pi _1(E)\rangle =\langle \pi _2(E)\rangle =\mathbb{P}^2$ and hence, since $\deg(E)=3$, $h^1(\mathcal{I}_E(1,0,0)) =  h^1(\mathcal{I}_E(0,1,0))  =0$. Fix a general $A'\in \mathcal{S}(Y',q')$ and set $A:= A'\cup \{p\}\in \mathcal{S}(Y_{2,2,1},q)$ (because we assume that $q$ has rank $3$). Assume by contradiction that $E$ is not of the form $B \cup \{p\}$, for some $B\in \mathcal{S}(Y',q')$. 
	
	Notice that, for a fiexed $E$, the generality of $A'$ implies that $A'\cap E=\emptyset$.
	Call $S:= A\cup E$ and set $\{H\}:= |\mathcal{I}_o(0,0,1)|$. Since $A'$ is a solution of $q'$ then $\pi_3(A)=\{o\}$, therefore $A'\subset H$. Moreover, since $\pi_3(p)\neq \pi_3(o)$, then $A\cap H =A'$ and concision gives $A\nsubseteq H$. 
	The residue of $S$ with respect to $H$ is $S\setminus S\cap H=\{p\}\cup (E\setminus E\cap H)$ and since $S\nsubseteq H $, by  \cite[Lemma 1.13]{BBS} either $h^1(\mathcal{I}_{(E\setminus E\cap H)\cup \{p\}}(1,1,0)) >0$ or $E\setminus E\cap H =\{p\}$.
	\begin{itemize}
		\item Assume $h^1(\mathcal{I} _{(E\setminus E\cap H)\cup \{p\}}(1,1,0)) >0$. At the beginning of this proof we have already remarked that if $p\in E$ then $h^1(\mathcal{I}_E(1,0,0)) =h^1(\mathcal{I}_E(0,1,0))$ $ =0$ and $\deg(E)=3$; by this reason it is not possible that $h^1(\mathcal{I} _{(E\setminus E\cap H)\cup \{p\}}(1,1,0)) >0$. Thus the assumption $h^1(\mathcal{I} _{(E\setminus E\cap H)\cup \{p\}}(1,1,0)) >0$ implies that $p\notin E$. 
		Even if $p\notin E$ we do not know if for example $\pi _2(p) \in \pi_2(E)$ or not. 
		
		Assume for the moment that $\pi _2(p)\in \pi_2(E)$ and, to fix the ideas, write $E=\{u,v,w\}$ with $\pi _2(u)=\pi _2(p)$. Take $M\in |\mathcal{I}_{\{u,v\}}(0,1,0)|$. We have $S\cap (H\cup M) =S\setminus \{w\}$, because we remark that $H\supset A'$.
		Since $h^1(\mathcal{I} _w(1,0,0)) =h^1(\mathcal{I}_{\Res_{H\cup M}(S) }(1,0,0))=0$, by \cite[Lemma 1.13]{BBS} we would have that $w\in H\cup M$ which is a contradiction. 
		
		So $\pi_2(p)$ cannot belong to $\pi_2(E)$; but if this is the case, a general $D\in |\mathcal{I}_p(0,1,0)|$ does not intersect $E$ since $\mathcal{O}_{\mathbb{P}^2}(1)$ is very ample. 
		Thus $S\setminus S\cap (H\cup D) \ne \emptyset$ and moreover $S\setminus S\cap (H\cup D)\subseteq E$. As before, since $h^1(\mathcal{I}_E(1,0,0)) =0$, by \cite[Lemma 1.13]{BBS} we get a contradiction. Therefore it is absurd both that $\pi_2(p)\in \pi_2(E) $ and that $\pi_2(p)\notin \pi_2(E)$, so we have to conclude that also the hypothesis $h^1(\mathcal{I} _{(E\setminus E\cap H)\cup \{p\}}(1,1,0)) >0$ was absurd.
		
		\item  Assume now that $E\setminus E\cap H =\{p\}$, i.e. assume $E =\{p\}\cup E'$ with $E'\subset H$ and $\deg(E')=2$.  Note that $S\setminus (S\cap H)=\Res_H(S)=\{p\}$ and that $h^1(\mathcal{I} _p(1,1,0) =0$. Hence, by \cite[Lemma 1.13]{BBS}, we get that $S\subset H$ and therefore we get a contradiction with the autarky assumption because the minimal multiprojective space containing $q$ is $\mathbb{P}^2\times \mathbb{P}^2\times \mathbb{P}^1$. Therefore it is also not possible that $E=\{p\}\cup E'$ with $E'\subset H$.   
	\end{itemize}
	Thus $E$ is of type $\{p\}\cup A$ for some $A\in \mathcal{S}(Y',q')$ and this concludes the proof of the claim.
\end{proof}

Now we are ready to present the new statement of \cite[Proposition 3.10]{BBS}.
\begin{proposition}(\protect{\cite[Proposition 3.10]{BBS} revised})\label{prop:new statement}
	Let $Y':=\mathbb{P}^1\times \mathbb{P}^1\times \{u_3\}\times \cdots \times \{ u_k\}$ be a proper subset of $Y_{n_1,\dots,n_k}=\mathbb{P}^{n_1}\times \cdots \times \mathbb{P}^{n_k}$ where we assume either $k\geq 4$ or $k=3$ and $(n_1,n_2,n_3)\neq (2,1,1)$. Take $q'\in \langle \nu(Y_{n_1,\dots,n_k})\setminus \nu(Y')\rangle $, $A\in \mathcal{S}(Y',q')$ and $p\in Y_{n_1,\dots,n_k}\setminus Y'$. Assume that $Y_{n_1,\dots,n_k}$ is the minimal multiprojective space containing $A\cup \{p\}$ and take $q\in \langle \{q',\nu(p) \} \rangle \setminus \{q',\nu(p)\}$.
	\begin{enumerate}
		\item $\sum_{i=1}^k n_i\geq4$;
	 $n_1,n_2\leq 2$, $n_3, \ldots , n_k\leq 1$;
 if $k\geq 3 $ then $ r_{\nu (Y_{n_1,\dots, n_k})}(q)>1$.
	   \item\label{f2} $r_{\nu(Y_{n_1,\dots ,n_k})}(q)=3$ and $\mathcal{S}(Y_{n_1,\dots ,n_k},q)=\{ \{p\} \cup A\}_{A\in \mathcal{S}(Y',q')}$.
	
	\item\label{f3} $\nu(Y_{n_1,\dots ,n_k})$ is the concise Segre of $ q$.
	
		\end{enumerate}
\end{proposition}

\begin{proof}
The proof of \cite[Proposition 3.10]{BBS} is splitted in two cases depending on whether $Y_{n_1,\dots,n_k}$ is made by all projective lines or not and both cases are worked out by induction. If $(n_1,\dots,n_k)=(1,\dots,1)$ the induction is contained in steps (B) and (C) of the proof of \cite[Proposition 3.10]{BBS} and they are not altered by the new statement. If instead $Y_{n_1,\dots,n_k}$ contains at least one projective plane, then we need to use $Y_{n_1,n_2,n_3,n_4}=\mathbb{P}^2\times \mathbb{P}^1\times \mathbb{P}^1\times \mathbb{P}^1$ nstead of $\mathbb{P}^2\times \mathbb{P}^1\times \mathbb{P}^1$ as base of the induction for which step (D) will then act as the inductive step.
	Case $\mathbb{P}^2\times \mathbb{P}^1\times \mathbb{P}^1\times \mathbb{P}^1$ follows from the case $\mathbb{P}^1\times \mathbb{P}^1\times \mathbb{P}^1\times \mathbb{P}^1$ proved in step (C) as follow. Consider a general $u\in \mathbb{P}^2$ and the linear projection $\mathbb{P}^2\setminus \{u\}\rightarrow  \mathbb{P}^1$. Construct
	the associate morphism $(\mathbb{P}^2\setminus \{u\})\times \mathbb{P}^1\times \mathbb{P}^1\times \mathbb{P}^1\to  \mathbb{P}^1\times \mathbb{P}^1\times \mathbb{P}^1\times \mathbb{P}^1$ and consider the projection from $\Lambda = \nu (\{u\}\times \mathbb{P}^1\times \mathbb{P}^1\times \mathbb{P}^1)$
	as in step (D).
This covers the proof of Proposition \ref{prop:new statement} for the case $k\geq 4$. Since case $k=3$ is completely covered by Lemma \ref{k=3unico} this concludes the proof of the statement.
\end{proof}

\begin{remark}
	The only statement in the rest of \cite{BBS} citing \cite[Proposition 3.10]{BBS} is Proposition 5.1 but the result is not altered using the revised Proposition \ref{prop:new statement}.
\end{remark}

With the above result we completely covered Proposition \ref{prop:new statement}. Now Theorem \ref{newMAINTHM} is completely fixed but for the sake of completeness let us show that the case $(n_1,n_2,n_3)=(2,1,1)$ fits only inside items \ref{3.} and \ref{4.}. 
 In this case the corresponding tensor space $\mathbb{P}(\mathbb{C}^3\otimes \mathbb{C}^2\otimes \mathbb{C}^2)$ has a finite number of orbits with respect to the action of $\mathrm{Aut}(\mathbb{P}^2)\times \mathrm{Aut}(\mathbb{P}^1)\times \mathrm{Aut}(\mathbb{P}^1)$ (cf. \cite{Par}, also  \cite[Table 1]{BL}) and there are only two possibilities for a concise rank-3 tensor, namely cases 7 and 8 of \cite[Table 1]{BL}. We already proved in Corollary \ref{diffcases_pencil=} that case 7 corresponds to \cite[Example 3.7]{BBS}, while in Corollary  \ref{cor:caso8bl} we saw that case 8 corresponds to \cite[Example 3.6]{BBS}. 
 
 We see now how to distiguish these two cases in a more geometrical way. 
\smallskip

Let $q$ be a rank-3 tensor in $\langle X_{2,1,1} \rangle $ and fix a solution $A\in \mathcal{S}(Y_{2,1,1},q)$.  Since $\#A=3$ and $h^0(\mathcal{O}_{Y_{2,1,1}}(0,1,1))$ $=4$, there is $G\in |\mathcal{I}_A(0,1,1)|$. The strength of the next claim is that we can prove that in this particular instance the space of solutions of $q$ with respect to $Y_{2,1,1}$ coincides with the space of solutions of $q$ with respect to $G$.
\begin{claim}\label{claim:sp_sol_coincide} Let $q\in \langle X_{2,1,1} \rangle $. Every $B\in \mathcal{S}(Y_{2,1,1},q)$ is contained in G and hence $\mathcal{S}(Y_{2,1,1},q)=\mathcal{S}(G,q) $.
\end{claim}
\begin{proof}
 Fix $B\in \mathcal{S}(Y_{2,1,1},q)$. Since $A\subset G$ the statement for $B=A$ is trivial, so let us assume $B\ne A$ and set $S:= A\cup B$.  Since $Y_{2,1,1}$ is the minimal multiprojective space containing $B$ then $\langle \pi_1(B)\rangle =\mathbb{P}^2$ and hence $h^1(\mathcal{I} _B(1,0,0)) =0$. Moreover, notice that $S\setminus S\cap G\subseteq B$ and therefore we have that $h^1(\mathcal{I} _{S\setminus S\cap G}(1,0,0))=0$.
Thus by \cite[Lemma 1.13]{BBS} we have that $B\subset G$.
\end{proof}
	Every $G\in |\mathcal{O}_{Y_{2,1,1}}(0,1,1)|$ is of the form $G =\mathbb{P}^2\times C$ for some $C\in |\mathcal{O}_{\mathbb{P}^1\times \mathbb{P}^1}(1,1)|$ and viceversa. Since $C$ is a hyperplane section of a smooth quadric in the Segre embedding of the last two factors $\mathbb{P}^1\times \mathbb{P}^1$ of $Y_{2,1,1}$ then either $C$ is a smooth conic or $C=L\cup R$ with $L\in |\mathcal{O}_{\mathbb{P}^1\times \mathbb{P}^1}(1,0)|$, $R\in |\mathcal{O}_{\mathbb{P}^1\times \mathbb{P}^1}(0,1)|$ and $L\cap R$ is a unique point $o\in \mathbb{P}^1\times \mathbb{P}^1$.
Let us distinguish two cases depending on wether $G$ is irreducible or not.
\begin{enumerate}
	\item Fix a solution $A$ such that $G$ is irreducible, i.e. assume that $C$ is irreducible and hence smooth. Let $u_i:\mathbb{P}^1\times \mathbb{P}^1\to \mathbb{P}^1$ for $i=1,2$ denote the projection from the last two factors of $Y_{2,1,1}$ onto the second and third factor of $Y_{2,1,1}$ respectively. Note that each $u_{i|C}: C\to \mathbb{P}^1$ has degree $1$ and hence it is an isomorphism. Claim \ref{claim:sp_sol_coincide} shows that $\#\pi _2(B)=\#\pi _3(B)=3$ for all $B\in \mathcal{S}(Y_{2,1,1},q)$. Taking as $A$ the union of $3$ general points of $Y_{2,1,1}$ we see that this case occurs. Moreover, the open orbit of $\sigma_3(X_{2,1,1})$ arises here and by Claim \ref{claim:sp_sol_coincide} this is the only case in which we fall in this orbit. 
	
	The case just described is \cite[Example 3.6]{BBS} with the additional observation that $\mathcal{S}(Y_{2,1,1},q)=\mathcal{S}(G,q)$.
	
	\item Fix $A$ such that $G$ is reducible and write $G=G_1\cup G_2$ with $G_1=\mathbb{P}^2\times L$, $G_2 =\mathbb{P}^2\times R$ and $G_1\cap G_2 = \mathbb{P}^2\times \{o\}$. This case is precisely the case described in \cite[Example 3.7 and Proposition 3.5]{BBS} with the additional information that $\mathcal{S}(Y_{2,1,1},q) =\mathcal{S}(G,q)$. Since $Y_{2,1,1}$ is the minimal multiprojective space containing $A$ then $A\nsubseteq G_1$ and $A\nsubseteq G_2$. We have $\#(A\cap (G_1\cap G_2))\le 1$ and $1 \le \#(A\cap G_i)\le 2$ for $i=1,2$, and moreover $$\#(G_1\cap A)+\#(G_2\cap A)=3+\#(A\cap G_1\cap G_2).$$  
	  Notice that both $\pi _2(G_1\cap A)$ and $\pi _3(G_2\cap A)$ is a single point and hence at least one $i\in \{2,3\}$ has $\#\pi _i(A)=2$.
	  Let us treat the two cases separately.
	  \begin{itemize}
	\item The case $\#\pi _2(A)=\#\pi_3(A)=2$ occurs if and only if $A\cap G_1\cap G_2\ne \emptyset$, i.e. if and only if the projection of $A$ in the last two factors contains $\{o\}=L\cap R$. To fix the ideas denote by $A=\{a,b,c\}$, with $a,b\in G_1$ and $c \in G_2$ and set $L=\{o_L\}\times \mathbb{P}^1$, $R=\mathbb{P}^1\times \{o_R\}$, where $o=(o_L,o_R)$.
	  In this case $\#\pi_2(A)=\#\pi_3(A)=2$ where $\{o_L\}\in \pi_2(A)$ and $\{o_R\}\in \pi_3(A)$ and either $a$ is of the form $(\pi_1(a), o_L,o_R)$ or $b$ is of the form $(\pi_1(b),o_L,o_R)$.
	   \item Taking as $A$ a general union of two general points of $G_1$ and a point of $G_2$ (or viceversa), we see that also the case $\#\pi_2(A)=2$ and $\#\pi _3(A) =3$ (or $\#\pi_2(A)=2$ and $\#\pi _3(A) =3$) occurs.
	Thus $\mathcal{S}(Y_{2,1,1},q)$ has precisely $2$ irreducible components, as observed in \cite[Example 3.7]{BBS}, and the dimension of the space of solution is $\dim \mathcal{S}(Y_{2,1,1},q)=4$.
	\end{itemize}
\end{enumerate}
From the above discussion we see that a single $A\in \mathcal{S}(Y_{2,1,1},q)$ is sufficient to know if $q$ is in the open orbit of $\sigma_3(X_{2,1,1})$ of case $8$ of \cite[Table 1]{BL} or in the smaller orbit of case $7$ of \cite[Table 1]{BL}.

For the sake of clarity we conclude by summing up the above discussion in the following statement.
\begin{proposition}
Let $q\in \langle X_{2,1,1}\rangle$ be a concise rank-3 tensor and fix a solution $A\in \mathcal{S}(Y_{2,1,1},q)$. Then there exists $G\in \vert \mathcal{O}_{Y_{2,1,1}}(0,1,1) \vert$ such that $\mathcal{S}(Y_{2,1,1})=\mathcal{S}(G,q)$ and either $G$ is irreducible, $\dim(\mathcal{S}(Y_{2,1,1},q))=3$ and $q$ is as in \cite[Example 3.6]{BBS}, or $G$ is reducible, $\dim(\mathcal{S}(Y_{2,1,1},q))=4$ and $q$ is as in \cite[Example 3.7]{BBS}.
\end{proposition}


\bibliographystyle{alpha}
\bibliography{bibliografia}

\newcommand{\etalchar}[1]{$^{#1}$}
\begin{thebibliography}{DLDMV00}

\bibitem[AGH{\etalchar{+}}14]{AetAl}
A.~Anandkumar, R.~Ge, D.~Hsu, S.~M. Kakade, and M.~Telgarsky.
\newblock Tensor decompositions for learning latent variable models.
\newblock {\em Journal of machine learning research}, 15:2773--2832, 2014.

\bibitem[AHJK13]{Rcit}
A.~Anandkumar, D.~J. Hsu, M.~Janzamin, and S.~M. Kakade.
\newblock When are overcomplete topic models identifiable? uniqueness of tensor
  tucker decompositions with structured sparsity.
\newblock {\em Advances in neural information processing systems}, 26, 2013.

\bibitem[AMR09]{AMR}
E.~S. Allman, C.~Matias, and J.~A. Rhodes.
\newblock Identifiability of parameters in latent structure models with many
  observed variables.
\newblock {\em Ann. Stat.}, 37:3099--3132, 2009.

\bibitem[AOP09]{AOP}
H.~Abo, G.~Ottaviani, and C.~Peterson.
\newblock Induction for secant varieties of segre varieties.
\newblock {\em Trans. Am. Math. Soc.}, 361:767--792, 2009.

\bibitem[APRS10]{APRS}
E.~S. Allman, S.~Petrovi{\'c}, J.~A. Rhodes, and S.~Sullivant.
\newblock Identifiability of two-tree mixtures for group-based models.
\newblock {\em IEEE/ACM transactions on computational biology and
  bioinformatics}, 8(3):710--722, 2010.

\bibitem[BB19]{BB3sec}
E.~Ballico and A.~Bernardi.
\newblock On the ranks of the third secant variety of segre-veronese
  embeddings.
\newblock {\em Linear Multilinear Algebra}, 67:583--597, 2019.

\bibitem[BBC18]{BBC}
E.~Ballico, A.~Bernardi, and L.~Chiantini.
\newblock On the dimension of contact loci and the identifiability of tensors.
\newblock {\em Arkiv f{\"o}r Matematik}, 56(2):265--283, 2018.

\bibitem[BBC{\etalchar{+}}19]{BBCMM}
E.~Ballico, A.~Bernardi, I.~Carusotto, S.~Mazzucchi, and V.~Moretti.
\newblock {\em Quantum physics and geometry}.
\newblock Springer, 2019.

\bibitem[BBS20]{BBS}
E.~Ballico, A.~Bernardi, and P.~Santarsiero.
\newblock Identifiability of rank-3 tensors.
\newblock {\em Mediterr. J. Math.}, 18:1--26, 2020.

\bibitem[BC12]{BCquantum}
A.~Bernardi and I.~Carusotto.
\newblock Algebraic geometry tools for the study of entanglement: an
  application to spin squeezed states.
\newblock {\em Journal of Physics A: Mathematical and Theoretical},
  45(10):105304, 2012.

\bibitem[BC13]{Bocci2011}
C.~Bocci and L.~Chiantini.
\newblock On the identifiability of binary {S}egre products.
\newblock {\em J. Algebraic Geom.}, 22:1--11, 2013.

\bibitem[BC19]{BClibro}
C.~Bocci and L.~Chiantini.
\newblock {\em An introduction to algebraic statistics with tensors}, volume~1.
\newblock Springer, 2019.

\bibitem[BCO14]{BCO}
C.~Bocci, L.~Chiantini, and G.~Ottaviani.
\newblock Refined methods for the identifiability of tensors.
\newblock {\em Ann. di Mat. Pura ed Appl.}, 193:1691--1702, 2014.

\bibitem[BL13]{BL}
J.~Buczy{\'n}ski and J.~M. Landsberg.
\newblock Ranks of tensors and a generalization of secant varieties.
\newblock {\em Linear Algebra Appl.}, 438:668--689, 2013.

\bibitem[Bor13]{Boralevi}
A.~Boralevi.
\newblock A note on secants of grassmannians.
\newblock {\em Rend. Istit. Mat. Univ. Trieste}, 45:67--72, 2013.

\bibitem[BV18]{BV}
A.~Bernardi and D.~Vanzo.
\newblock A new class of non-identifiable skew-symmetric tensors.
\newblock {\em Annali di Matematica Pura ed Applicata (1923-)},
  197(5):1499--1510, 2018.

\bibitem[CC06]{CC2}
L.~Chiantini and C.~Ciliberto.
\newblock On the concept of k-secant order of a variety.
\newblock {\em J. London Math. Soc}, 73:436--454, 2006.

\bibitem[CK11]{CK}
E.~Carlini and J.~Kleppe.
\newblock Ranks derived from multilinear maps.
\newblock {\em J. Pure Appl. Algebra}, 215:1999--2004, 2011.

\bibitem[CM21]{CM21}
A.~Casarotti and M.~Mella.
\newblock Tangential weak defectiveness and generic identifiability.
\newblock {\em International Mathematics Research Notices}, 06 2021.

\bibitem[CM22]{CM22}
A.~Casarotti and M.~Mella.
\newblock From non-defectivity to identifiability.
\newblock {\em Journal of the European Mathematical Society}, 2022.

\bibitem[CO12]{CO}
L.~Chiantini and G.~Ottaviani.
\newblock On generic identifiability of 3-tensors of small rank.
\newblock {\em SIAM J. Matrix Anal. Appl.}, 33:1018--1037, 2012.

\bibitem[COV14]{COV14}
L.~Chiantini, G.~Ottaviani, and N.~Vannieuwenhoven.
\newblock An algorithm for generic and low-rank specific identifiability of
  complex tensors.
\newblock {\em SIAM J. Matrix Anal. Appl.}, 35:1265--1287, 2014.

\bibitem[COV17]{COV17}
L.~Chiantini, G.~Ottaviani, and N.~Vannieuwenhoven.
\newblock Effective criteria for specific identifiability of tensors and forms.
\newblock {\em SIAM J. Matrix Anal. Appl.}, 38:656--681, 2017.

\bibitem[DDL13]{DD1}
I.~Domanov and L.~De~Lathauwer.
\newblock On the uniqueness of the canonical polyadic decomposition of
  third-order tensors---part i: Basic results and uniqueness of one factor
  matrix.
\newblock {\em SIAM J. Matrix Anal. Appl.}, 34:855--875, 2013.

\bibitem[DDL14]{DL14}
I.~Domanov and L.~De~Lathauwer.
\newblock Canonical polyadic decomposition of third-order tensors: Reduction to
  generalized eigenvalue decomposition.
\newblock {\em SIAM J. Matrix Anal. Appl.}, 35:636--660, 2014.

\bibitem[DLDMV00]{flat}
L.~De~Lathauwer, B.~De~Moor, and J.~Vandewalle.
\newblock A multilinear singular value decomposition.
\newblock {\em SIAM J. Matrix Anal. Appl.}, 21:1253--1278, 2000.

\bibitem[Gan59]{GantV12}
F.~R. Gantmacher.
\newblock {\em The theory of matrices. {V}ols. 1, 2}.
\newblock Chelsea Publishing Co., New York, 1959.
\newblock Translated by K. A. Hirsch.

\bibitem[GKZ08]{GKZ}
I.~M. Gelfand, M.~M. Kapranov, and A.~V. Zelevinsky.
\newblock {\em Discriminants, resultants and multidimensional determinants}.
\newblock Modern Birkh\"{a}user Classics. Birkh\"{a}user Boston, Inc., Boston,
  MA, 2008.
\newblock Reprint of the 1994 edition.

\bibitem[GM19]{GM19}
F.~Galuppi and M.~Mella.
\newblock Identifiability of homogeneous polynomials and cremona
  transformations.
\newblock {\em Journal f{\"u}r die reine und angewandte Mathematik (Crelles
  Journal)}, 2019(757):279--308, 2019.

\bibitem[Gri78]{grigoriev}
D.~Y. Grigoriev.
\newblock Multiplicative complexity of a pair of bilinear forms and of the
  polynomial multiplication.
\newblock In {\em International Symposium on Mathematical Foundations of
  Computer Science}, pages 250--256. Springer, 1978.

\bibitem[GS]{M2}
Daniel~R. Grayson and Michael~E. Stillman.
\newblock Macaulay2, a software system for research in algebraic geometry.
\newblock Available at \url{http://www.math.uiuc.edu/Macaulay2/}.

\bibitem[Har77]{Hart}
R.~Hartshorne.
\newblock Graduate texts in mathematics.
\newblock {\em Algebraic Geometry}, 52, 1977.

\bibitem[Hit28]{hit1928}
F.~L. Hitchcock.
\newblock Multiple invariants and generalized rank of a p-way matrix or tensor.
\newblock {\em J. Math. Phys.}, 7:39--79, 1928.

\bibitem[HOOS19]{HOOShomotopy}
J.~D. Hauenstein, L.~Oeding, G.~Ottaviani, and A.~J. Sommese.
\newblock Homotopy techniques for tensor decomposition and perfect
  identifiability.
\newblock {\em J. f{\"u}r Reine Angew. Math.}, 2019:1--22, 2019.

\bibitem[J{\'a}J79]{jaja}
J.~J{\'a}J{\'a}.
\newblock Optimal evaluation of pairs of bilinear forms.
\newblock {\em SIAM J. Comput.}, 8:443--462, 1979.

\bibitem[Kac80]{K}
V.~G. Kac.
\newblock Some remarks on nilpotent orbits.
\newblock {\em J. Algebra}, 64:190--213, 1980.

\bibitem[KB09]{Kolda}
T.~G. Kolda and B.~W. Bader.
\newblock Tensor decompositions and applications.
\newblock {\em SIAM review}, 51(3):455--500, 2009.

\bibitem[Kru77]{K77}
J.~B. Kruskal.
\newblock Three-way arrays: rank and uniqueness of trilinear decompositions,
  with application to arithmetic complexity and statistics.
\newblock {\em Linear Algebra Appl.}, 18:95--138, 1977.

\bibitem[Lan12]{Lands}
J.~M. Landsberg.
\newblock {\em Tensors: geometry and applications}, volume 128 of {\em Graduate
  Studies in Mathematics}.
\newblock American Mathematical Society, Providence, RI, 2012.

\bibitem[LM04]{LandMan}
J.~M. Landsberg and L.~Manivel.
\newblock On the ideals of secant varieties of {S}egre varieties.
\newblock {\em Found. Comput. Math.}, 4:397--422, 2004.

\bibitem[LMR22]{LMR22}
A.~Laface, A.~Massarenti, and R.~Rischter.
\newblock On secant defectiveness and identifiability of segre--veronese
  varieties.
\newblock {\em Revista Matem{\'a}tica Iberoamericana}, 2022.

\bibitem[LP21]{Lovitz}
B.~Lovitz and F.~Petrov.
\newblock A generalization of kruskal's theorem on tensor decomposition.
\newblock {\em arXiv preprint arXiv:2103.15633}, 2021.

\bibitem[MMS18]{MMS}
A.~Massarenti, M.~Mella, and G.~Staglian{\`o}.
\newblock Effective identifiability criteria for tensors and polynomials.
\newblock {\em Journal of Symbolic Computation}, 87:227--237, 2018.

\bibitem[Oed11]{OedingTan}
L.~Oeding.
\newblock Set-theoretic defining equations of the tangential variety of the
  segre variety.
\newblock {\em J. Pure Appl. Algebra}, 215:1516--1527, 2011.

\bibitem[Par01]{Par}
P.~G. Parfenov.
\newblock Orbits and their closures in the spaces
  $\mathbb{C}^{k_1}\otimes\dots\otimes \mathbb{C}^{k_r}$.
\newblock {\em Mat. Sb.}, 192:89--112, 2001.

\bibitem[Qi13]{Qieqsigma3}
Y.~Qi.
\newblock Equations for the third secant variety of the segre product of n
  projective spaces.
\newblock {\em arXiv preprint arXiv:1311.2566}, 2013.

\bibitem[RS12]{RS}
J.~A. Rhodes and S.~Sullivant.
\newblock Identifiability of large phylogenetic mixture models.
\newblock {\em Bulletin of mathematical biology}, 74(1):212--231, 2012.

\bibitem[San22]{phdtesi}
P.~Santarsiero.
\newblock {\em Identifiability of small rank tensors and related problems}.
\newblock PhD thesis, Università di Trento, Italy, 2022.

\bibitem[SDL15]{SDL15}
M.~S{\o}rensen and L.~De~Lathauwer.
\newblock New uniqueness conditions for the canonical polyadic decomposition of
  third-order tensors.
\newblock {\em SIAM J. Matrix Anal. Appl.}, 36:1381--1403, 2015.

\bibitem[Tei86]{teichert}
L.~Teichert.
\newblock {\em Die Komplexit{\"a}t von Bilinearformpaaren {\"u}ber beliebigen
  K{\"o}rpern}.
\newblock na, 1986.

\end{thebibliography}

\end{document}